\documentclass[a4paper,10pt,leqno]{amsart}
\usepackage{amsmath,amsfonts,amsthm,amssymb,dsfont}

\usepackage[T1]{fontenc}
\usepackage{enumerate}
\usepackage{mathrsfs}
\usepackage{mathtools}
\usepackage{enumerate}
\usepackage{caption}
\usepackage{graphicx}
\usepackage{overpic}

\usepackage[dvipsnames]{xcolor}
\usepackage{quiver}
\usepackage[alphabetic]{amsrefs}
\usepackage{float}

\newtheorem{thm}{Theorem}[section]
\newtheorem{lem}[thm]{Lemma}
\newtheorem{lem-defin}[thm]{Lemma-Definition}
\newtheorem{prop}[thm]{Proposition}

\newtheorem{theoremalph}{Theorem}

\newtheorem*{propC}{Proposition C}

\theoremstyle{definition}
\newtheorem{rem}[thm]{Remark}
\newtheorem{question}[thm]{Question}
\newtheorem{defin}[thm]{Definition}

\newcommand{\Z}{\mathbb{Z}}

\newcommand{\Fix}{\mathrm{Fix}}

\newcommand{\bbN}{{\mathbb N}}

\newcommand{\cG}{{\mathcal G}}

\newcommand{\cK}{{\mathcal K}}

\newcommand{\cT}{{\mathcal T}}

\newcommand{\link}{\mathrm{Lk}}
\newcommand{\Axis}{\mathrm{Axis}}

\makeatletter
\def\paragraph{\@startsection{paragraph}{4}%
  \z@\z@{-\fontdimen2\font}%
  {\normalfont\bfseries}}
\makeatother

\begin{document}

\title[Tits alternative for two-dimensional Artin groups]{ The Tits alternative for  two-dimensional Artin groups and Wise's power alternative}

\author{Alexandre Martin}

\maketitle

\begin{abstract}
We show that two-dimensional Artin groups satisfy a strengthening of the Tits alternative: their subgroups  either contain a non-abelian free group or are virtually  free abelian of rank at most $2$.

When in addition the associated Coxeter group is hyperbolic, we  answer in the affirmative a question of Wise on the subgroups generated by large powers of two elements: given any two  elements $a, b$ of a two-dimensional Artin group of hyperbolic type, there exists an integer $n\geq 1$ such that  $a^n$ and $b^n$ either commute or generate a non-abelian free subgroup. 
\end{abstract}

\tableofcontents

\section{Introduction}

\paragraph{Motivation and statement of results.} The Tits Alternative was first introduced by Tits in \cite{TitsAlternative}, where he proved that a finitely generated subgroup of a linear group over a field either contains a free group or is virtually solvable. This striking dichotomy has since been established for a large class of groups. A meta-conjecture asserts that groups that are non-positively curved in a suitable sense satisfy the Tits Alternative. While this has been proved for certain classes of groups (CAT(0) cubical groups \cite{SW, CS}, hierarchically hyperbolic groups \cite{TitsHHG, TitsHHGcorr}, etc.) it remains open in general, and even for CAT(0) groups despite some recent progress~\cite{OP}. 

In this article, we will focus on Artin (or Artin-Tits) groups, a large class of groups introduced by Tits in \cite{TitsArtin}. These groups generalise braid groups  and have strong connections with Coxeter groups. While the geometry and structure of Coxeter groups are now well-understood, the geometry of Artin groups remains elusive in general. It is conjectured that these groups are CAT(0), although this is only known in very few cases (see for instance \cite[Conjecture~A]{HaettelXXL} and the discussion thereafter for a list of partial results). Nonetheless, certain classes of Artin groups have been shown to be non-positively curved in a more general sense, see for instance \cite{HO1, HO2}. 

 In light of the above, it is natural to ask whether all Artin groups satisfy the Tits Alternative. This is known already for certain classes of Artin groups, including Artin groups of spherical type \cite{CW, Digne}, Artin groups of FC type \cite{MP1}, and two-dimensional Artin groups of hyperbolic type \cite{MP2}. The first goal of this paper is to generalise the latter result to all two-dimensional Artin groups:

\begin{theoremalph}\label{thm:Tits_dim2}
Let $A_\Gamma$ be a two-dimensional Artin group. Then every subgroup of $A_\Gamma$ either contains a non-abelian free group or  is virtually free abelian of rank at most $2$. 
\end{theoremalph}

When proving the Tits Alternative, one is led to construct free subgroups of the group under study, which is challenging even for CAT(0) groups.   However, the situation sometimes becomes much more manageable when looking at the subgroups generated by \textit{suitably large powers}. In the case of a hyperbolic group for instance,  it is a standard result that given two infinite order elements $a, b$ with disjoint limit sets in the Gromov boundary of the group, there exists a positive integer $n$ such that $a^n$ and $b^n$ generate a free group. 
 
For groups that are non-positively curved in a more general sense (including CAT(0) groups, biautomatic groups, etc.), Wise asked in \cite[Question 2.7]{BestvinaQuestions} whether a similar ``taming'' phenomenon occurs, namely: given any two elements $a, b$, does there exists a positive integer $n$ such that $a^n$ and $b^n$ either generate a free group or a free abelian group? We will say that a group satisfy \textbf{Wise's Power Alternative} when such a result holds. Leary--Minasyan's recent example of a CAT(0) that is not biautomatic \cite{LearyMinasyan} also provides the first example of a non-positively curved group that does not satisfy this alternative. It is thus natural to ask the following:

\begin{question}
	Which (non-positively curved) groups satisfy Wise's Power Alternative?
	\end{question}

There are already a few known groups that satisfy this alternative. For instance, Baudisch showed \cite{Baudisch} that right-angled Artin groups satisfy an even stronger statement: two elements in a right-angled Artin group either commute or generate a free group. In particular, Wise's Power Alternative is satisfied by all groups that virtually embed in a right-angled Artin group, such as Coxeter groups \cite{HaglundWise}. In another direction, Koberda proved that  two arbitrary elements in a non-exceptional  mapping class group admit powers that are contained in a right-angled Artin subgroup \cite{Koberda}. In particular, Wise's Power Alternative also holds for non-exceptional mapping class groups. 

Since we believe Artin groups to be non-positively curved in general, we ask the following: 

\begin{question}
	Which Artin groups satisfy Wise's Power Alternative?
	\end{question}

The second main result of this article is the following:

\begin{theoremalph}\label{thm:power_subgroup}
	Two-dimensional Artin groups of hyperbolic type satisfy Wise's Power Alternative.
\end{theoremalph}

Note that XXL-type Artin groups are known to be CAT(0) \cite{HaettelXXL}, and that extra-large type Artin groups are known to be systolic \cite{HO2}, so the above result provides a large new class of non-positively curved groups that satisfy Wise's Power Alternative.
In a forthcoming article with Mark Hagen, we will prove Wise's Power Alternative for some other classes of Artin groups (alongside other  groups) via their actions on trees~\cite{HMtrees}.

\bigskip

\paragraph{Strategy of the proofs.} Our proof of the Tits Alternative is geometric and relies on the action of a two-dimensional Artin group on its  Deligne complex (see Definition~\ref{def:Deligne}), a particularly well-behaved CAT(0) simplicial complex. 

In \cite{MP2}, the authors could exploit the dynamics of the acylindrical action of two-dimensional  Artin groups \textit{of hyperbolic type} on a variation of their Deligne complex to construct many free subgroups. This approach is no longer possible for general two-dimensional Artin groups, and the strategy of this paper is therefore quite different.  Instead, given a subgroup $H$ of $A_\Gamma$, we will treat separate cases depending on whether the non-trivial elements of $H$ all act loxodromically on the Deligne complex, all act elliptically, or whether $H$ contains both loxodromic and elliptic elements. In particular, we provide a classification result for subgroups of $A_\Gamma$ (Proposition~\ref{prop:subgroup_dichotomy}). A key case to consider is the case where $H$ contains two elements that act as elliptic isometries of the Deligne complex  with disjoint stable fixed-point sets, where the \textit{stable} fixed-point set of an isometry is defined as the union of the fixed-point sets of all its non-trivial powers. We recall some general questions here:

\begin{question}
	Let $a$, $b$ be two elliptic isometries of a CAT(0) space $X$ that have disjoint stable fixed-point sets.
	\begin{itemize}
		\item[(i)] Does there exist an element of $\langle a, b \rangle$ that acts loxodromically on $X$?
		\item[(ii)] Does there exists a positive integer $n$ such that  $\langle a^n, b^n \rangle$ is a free group?
	\end{itemize}
\end{question}

Such questions have been studied in the literature, in particular in relation with the question of the existence of infinite torsion subgroups of CAT(0) groups (question $(i)$, see for instance \cite[Conjecture 1.5]{NOP}), and with the Tits Alternative (question $(ii)$). 
A key intermediate result in our proofs of Theorems~\ref{thm:Tits_dim2} and~\ref{thm:power_subgroup} is the following: 

\begin{propC}\label{thm:disjoint_elliptic}
	Let $A_\Gamma$ be a two-dimensional Artin group, and let $a$, $b$ be two elements of  $A_\Gamma$ that act elliptically on the Deligne complex $D_\Gamma$ with disjoint fixed-point sets. Then there exists a positive integer $n$ such that  $\langle a^n, b^n \rangle$ is a non-abelian free group. 
	
	Moreover, an element of $\langle a^n, b^n \rangle$ acts loxodromically on $D_\Gamma$ if and only if it is not conjugated to a power of $a^n$ or $b^n$. 
	\end{propC}

Our strategy to prove Proposition~C is to construct a tree embedded in the Deligne complex on which $\langle a^n, b^n \rangle$ acts and such that there is a bijection between edges of that tree and the reduced words on $\{a^{\pm n}, b^{\pm n}\}$ . To that end, we consider a geodesic $\gamma$ between the fixed-point sets of $a$ and $b$, and we show that for some large enough $n$, the geodesic segment $\gamma$ makes an angle at least $\pi$ with any of its $\langle a^n \rangle$-translates or $\langle b^n \rangle$-translates. The CAT(0) geometry of the Deligne complex then guarantees that the resulting subspace obtained as the union of all $\langle a^n, b^n \rangle$-translates of $\gamma$ is a convex subtree of~$D_\Gamma$. (see Figure~1)

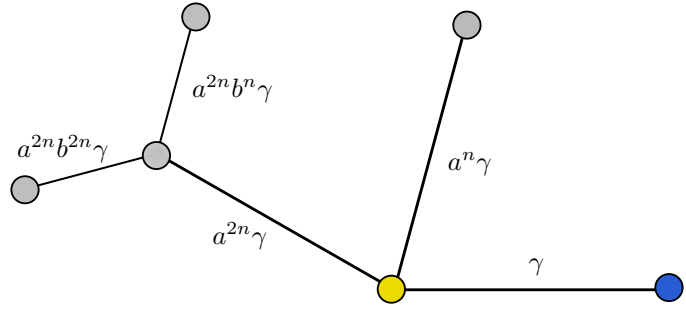
\begin{figure}\label{fig:tree}
	\begin{center}
			\begin{overpic}[width=0.7\textwidth]{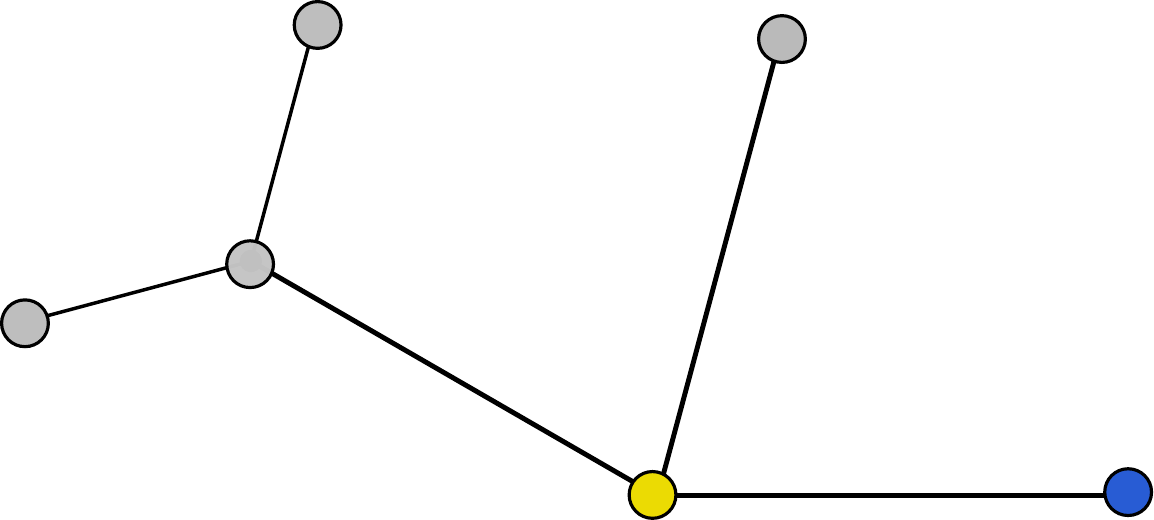}
				\put(77,5){$\gamma$}
				\put(65,20){$a^n\gamma$}
				\put(30,9){$a^{2n}\gamma$}
				\put(27,31){$a^{2n}b^n\gamma$}
				\put(1,22){$a^{2n}b^{2n}\gamma$}
			\end{overpic}
	\end{center}
\caption{A small portion representing the convex subtree of $D_\Gamma$ obtained by gluing together all the $\langle a^n, b^n\rangle$-translates of the geodesic segment $\gamma$ between  the fixed-point sets $\mathrm{Fix}(a)$ (in yellow) and~$\mathrm{Fix}(b)$ (in blue).}
\end{figure}

Controlling these angles requires a fine understanding of the local structure of the Deligne complex, and this problem can be translated into a problem about dihedral Artin groups, a class that is very well understood. 

\bigskip

In order to prove Theorem~\ref{thm:power_subgroup}, the additional assumption that $A_\Gamma$ is of hyperbolic type allows us to exploit the action of the Artin group on its \textit{coned-off Deligne complex} (see Definition~\ref{def:cone_off}), a variation of the Deligne complex introduced in \cite{MP2} on which $A_\Gamma$ acts acylindrically.  The additional dynamical properties of this action are a key tool in proving that certain subgroups generated by large powers are free.

\bigskip

\paragraph{Structure of the paper.} In Section~\ref{sec:Artin}, we recall standard definitions and results about Artin groups and their Deligne complexes. In Section~\ref{sec:elliptic}, we prove the key Proposition~C on  free subgroups generated by powers of elliptic elements with disjoint fixed-point sets. In Section~\ref{sec:Tits}, we prove Theorem~\ref{thm:Tits_dim2} by first providing a classification of the subgroups of a two-dimensional Artin group (Proposition~\ref{prop:subgroup_dichotomy}). Finally, in Section~\ref{sec:Wise}, we use the action of a two-dimensional Artin group of hyperbolic on its coned-off Deligne complex to prove Theorem~\ref{thm:power_subgroup}.

\bigskip

\paragraph{Aknowledgements} I would like to thank Piotr Przytycki, as this article grew out of previous work and discussions with him. I would also like to thank Thomas Delzant for providing a very short proof (and a much simpler one than the one originally there) of Lemma~\ref{lem:free_subgroup_mixed}. I also thank the anonymous referee for their comments and careful reading of this paper.

This work was partially supported by the EPSRC New Investigator Award EP/S010963/1.

\section{Preliminaries on Artin groups and their Deligne~complex} \label{sec:Artin}

This preliminary section contains definitions of the main groups and complexes used in this article, and recalls various results from the literature.

\subsection{Artin groups}

\begin{defin}
	A \textbf{presentation graph} is a finite simplicial graph $\Gamma$ such that every edge between vertices $a, b \in V(\Gamma)$ comes with a label $m_{ab} \geq 2$. The \textbf{Artin group} (or \textbf{Artin-Tits group}) associated to the presentation graph $\Gamma$ is the group $A_{\Gamma}$ given by 
the following presentation:
$$A_\Gamma \coloneqq \langle a \in V(\Gamma ) ~|~ \underbrace{aba\cdots}_{m_{ab}} = \underbrace{bab\cdots}_{m_{ab}}  ~~ 
\mbox{ whenever } a,b \mbox{ are connected by an edge of }\Gamma \rangle.$$

An Artin group on two generators $a,b$ with $m_{ab}<\infty$ is called a \textbf{dihedral} Artin group. 
Given an Artin group $A_\Gamma$, the \textbf{associated Coxeter group}  $W_\Gamma$ is obtained by adding the relations 
 $a^2 =1$ for every $a \in V(\Gamma)$ . An Artin group is said to be of \textbf{hyperbolic type} if the associated Coxeter group is hyperbolic, and of \textbf{spherical type} if the associated 
Coxeter group is finite. 
\end{defin}

\begin{defin}
	Given an induced subgraph $\Gamma'\subset \Gamma$, the subgroup of $A_{\Gamma}$ generated by the vertices of $\Gamma'$ 
	is called a \textbf{standard parabolic subgroup}. Such a subgroup is isomorphic to the Artin group $A_{\Gamma'}$ by \cite{vdL}, and 
conjugates of standard parabolic subgroups are called \textbf{parabolic subgroups}. 
\end{defin}

\begin{defin}
	An Artin group is said to be \textbf{two-dimensional} if for every induced triangle $\Gamma'$ of $\Gamma$, the corresponding standard parabolic subgroup $A_{\Gamma'}$ is not of spherical type. 
\end{defin}

\subsection{The (modified) Deligne complex}

We recall here some important complex on which an Artin group acts.

\begin{defin}[{\cite{CD}}]\label{def:Deligne} 
	The \textbf{(modified) Deligne complex} of a an Artin group $A_\Gamma$ is the  simplicial complex defined as follows: 
	\begin{itemize}
		\item Vertices correspond to left cosets of standard parabolic subgroups of spherical type.
		\item For every $g \in A_\Gamma$ and for every chain of induced subgraphs $\Gamma_0 \subsetneq \cdots \subsetneq \Gamma_k$ such that $A_{\Gamma_0}, \ldots, A_{\Gamma_k}$ are of spherical type, we put a $k$-simplex between the vertices $gA_{\Gamma_0}, \ldots, gA_{\Gamma_k}$.
	\end{itemize}
In other words, the (modified) Deligne complex $D_\Gamma$ is the geometric realisation of the poset of left cosets of standard parabolic subgroups of spherical type.

The group $A_\Gamma$ acts on  $D_\Gamma$ by left multiplication on left cosets.
\end{defin}

Let us recall a few elementary properties about this action. Note that since vertices of a simplex of $D_\Gamma$ correspond to left cosets of pairwise distinct subgroups, $A_\Gamma$ acts on $D_\Gamma$ without inversion: if an element of $A_\Gamma$ globally stabilises a simplex of $D_\Gamma$, it fixes it pointwise. As a consequence, the stabiliser of a simplex of $D_\Gamma$ corresponding to the chain $gA_{\Gamma_0} \subsetneq \cdots \subsetneq gA_{\Gamma_k}$ is the parabolic subgroup $gA_{\Gamma_0}g^{-1}$. 

 In the rest of this article, we will simply speak of the \textbf{Deligne complex} instead of the modified Deligne complex for simplicity.

\bigskip 

For two-dimensional Artin groups, there is a simple description of the vertices of the Deligne complex $D_\Gamma$:

\begin{lem}\label{lem:type_vertices}
	Let $A_\Gamma$ be a two-dimensional Artin group. Then vertices of $D_\Gamma$ are of the following type:
	\begin{itemize}
		\item left cosets of the form $gA_{\Gamma'} = g\{1\}$ where $\Gamma'$ is the empty subgraph. Such vertices are called \textbf{of type 0} and their stabiliser is trivial.
		\item left cosets of the form $gA_{\Gamma'} = g \langle a \rangle $ where $\Gamma' = \{a\}$ is a single vertex of $\Gamma$. Such vertices are called \textbf{of type 1} and their stabiliser is infinite cyclic.
		\item left cosets of the form $gA_{\Gamma'} = g\langle a, b \rangle$ where $\Gamma'$ is an edge between two vertices $a, b$ of $\Gamma$. Such vertices are called \textbf{of type 2} and their stabiliser is isomorphic to a dihedral Artin group.
	\end{itemize} 
In particular, the Deligne complex $D_\Gamma$ is a simplicial complex of dimension $2$. Moreover, it follows from the discussion above that the stabilisers of higher dimensional simplices of $D_\Gamma$ are as follows: 
\begin{itemize}
	\item The stabiliser of an edge of $D_\Gamma$ is either infinite cyclic or trivial.
	\item The stabiliser of a triangle of $D_\Gamma$ is trivial.
\end{itemize}
\end{lem}

While the geometry of Deligne complexes is still mysterious in general (and strongly related to the $K(\pi,1)$-conjecture for Artin groups, see \cite{CD}), the geometry of Deligne complexes for two-dimensional Artin groups is well-understood:

\begin{thm}[{\cite{CD}}]\label{thm:Charney_Davis}
	Let $A_\Gamma$ be a two-dimensional Artin group. Its Deligne complex can be endowed with a CAT(0) metric as follows:
	For a chain $$g\{1\} \subsetneq g\langle a \rangle \subsetneq g \langle a,b \rangle ~~~~~~~~~~~~~~ \mbox{ (for $g \in A_\Gamma$ and $a, b$ adjacent vertices of  $\Gamma$) }$$ we identify the corresponding triangle of $D_\Gamma$ with a  triangle of the Euclidean plane $\mathbb{E}^2$ with angles $\frac{\pi}{2m_{ab}}, \frac{\pi}{2}$, and $\frac{\pi}{2} - \frac{\pi}{2m_{ab}}$ at the vertices $g\langle a, b\rangle, g\langle a \rangle,  g \{1\}$ respectively, and such that the edge between $g\langle a \rangle$ and $ g \{1\}$ has length $1$.  
	
	This defines an $A_\Gamma$-equivariant piecewise-Euclidean structure on $D_\Gamma$ (in the sense of \cite[Definition~I.7A.9]{BH}), which turns $D_\Gamma$ into a CAT$(0)$ space. 
	
	If moreover $A_\Gamma$ is of hyperbolic type, then one can adapt this construction to endow $D_\Gamma$ with an $A_\Gamma$-equivariant piecewise-hyperbolic metric that is CAT$(-1)$. (See for instance \cite[Section~3.1]{MP2} for details in that case.)
\end{thm}

\subsection{Dihedral Artin groups and some associated graphs.} \label{sec:graph_dihedral}

Dihedral Artin groups appear as stabilisers of vertices of  type $2$ in the  Deligne complex, so we 
mention here some results about these groups and the links of type $2$ vertices.

\medskip

A dihedral Artin group on two standard generators $s, t$ will be denoted $A_{st}$ for simplicity (even though the group depends on the coefficient $m_{st}$). There are two types of dihedral Artin groups: If $m_{st}=2$, the group is isomorphic to $\Z^2$. We now focus on the structure of dihedral Artin groups where $m_{st}\geq 3$. The following result is well-known to experts, see for instance \cite{Crisp}.

\begin{lem}\label{lem:virtually splits}
	Let $A_{st}$ be a dihedral Artin group with $m_{st}\geq 3$. Then  $A_{st}$ contains a finite-index subgroup that is isomorphic to a direct product of the form $\Z\times F_n$. 
	\label{lem:dihedral_split}
\end{lem}

Dihedral Artin groups have a rich structure, and admit in particular a normal form called a \textit{Garside normal form}, which we now describe.

\begin{defin} For a dihedral Artin group with $m_{st}\geq 3$,  the \textbf{Garside element} $\Delta_{st}\in A_{st}$ is defined as follows:
	$$\Delta_{st} \coloneqq \underbrace{sts\cdots}_{m_{st}}= \underbrace{tst\cdots}_{m_{st}}.$$
\end{defin}

\begin{lem} \label{lem:garside_generators} For a dihedral Artin group with $m_{st}\geq 3$, we have the following:
	\begin{itemize}
		\item If $m_{st}$ is even, then $\Delta_{st}$ is a central element.
		\item If $m_{st}$ is odd, then $\Delta_{st}s= t\Delta_{st}$ and $\Delta_{st}t= s\Delta_{st}$
	\end{itemize}
\end{lem}

\begin{defin}[Atoms, left-weighted form]\label{def:atoms}
	An \textbf{atom} of $A_{st}$ is a strict subword of either $\underbrace{sts\cdots}_{m_{st}}$ or 
	$\underbrace{tst\cdots}_{m_{st}}$, that is, an alternating product of $s$ and $t$ on strictly fewer than $m_{st}$ letters. 
	The set of all atoms of $A_{st}$ is denoted $M$. A concatenation of the form $m_1\cdots m_k$, with all $m_i \in M$, is said to 
	be \textbf{left-weighted} if for every $i$ the last letter of $m_i$ coincides with the first letter of $m_{i+1}$.  
\end{defin}

\begin{thm}[Garside normal form {\cite{MairesseMatheus}}]
	Every element of $A_{st}$  can be written in a unique way as a product of the form $m_1\cdots m_k \Delta_{st}^\ell$, where $\ell \in \Z$, the $m_i's$ are atoms of $A_{st}$ and the concatenation $m_1\cdots m_k$ is left-weighted.\qed
\end{thm} 

Note that in Definition~\ref{def:atoms}, we only consider atoms that are positive words in $s, t$, as negative powers of the generators $s, t$ can be absorbed in powers of the Garside element in the Garside normal form above. 

Consider the Cayley graph $\mathrm{Cay}(A_{st}, M)$ for the generating set $M$  consisting of all atoms of $A_{st}$. It follows from Lemma~\ref{lem:garside_generators} that $\langle \Delta_{st} \rangle$ acts on the graph  $\mathrm{Cay}(A_{st}, M)$ by multiplication \textit{on the right}. This motivates the following definition:

\begin{defin}\label{def:quasi_tree_simplices}
	We denote by $\cT_{st}$ the quotient graph: 
	$$\cT_{st} \coloneqq \mathrm{Cay}(A_{st}, M) / \langle \Delta_{st} \rangle$$
	
	where $\langle \Delta_{st} \rangle$ acts on  $\mathrm{Cay}(A_{st}, M)$ by multiplication on the right.
	The action of $A_{st}$ on $\mathrm{Cay}(A_{st}, M)$ by left multiplication induces an action of $A_{st}$  
	on $\cT_{st}$.
\end{defin}

\begin{figure}\label{fig:bestvina}
	\begin{center}
		\begin{overpic}[width=0.9\textwidth]{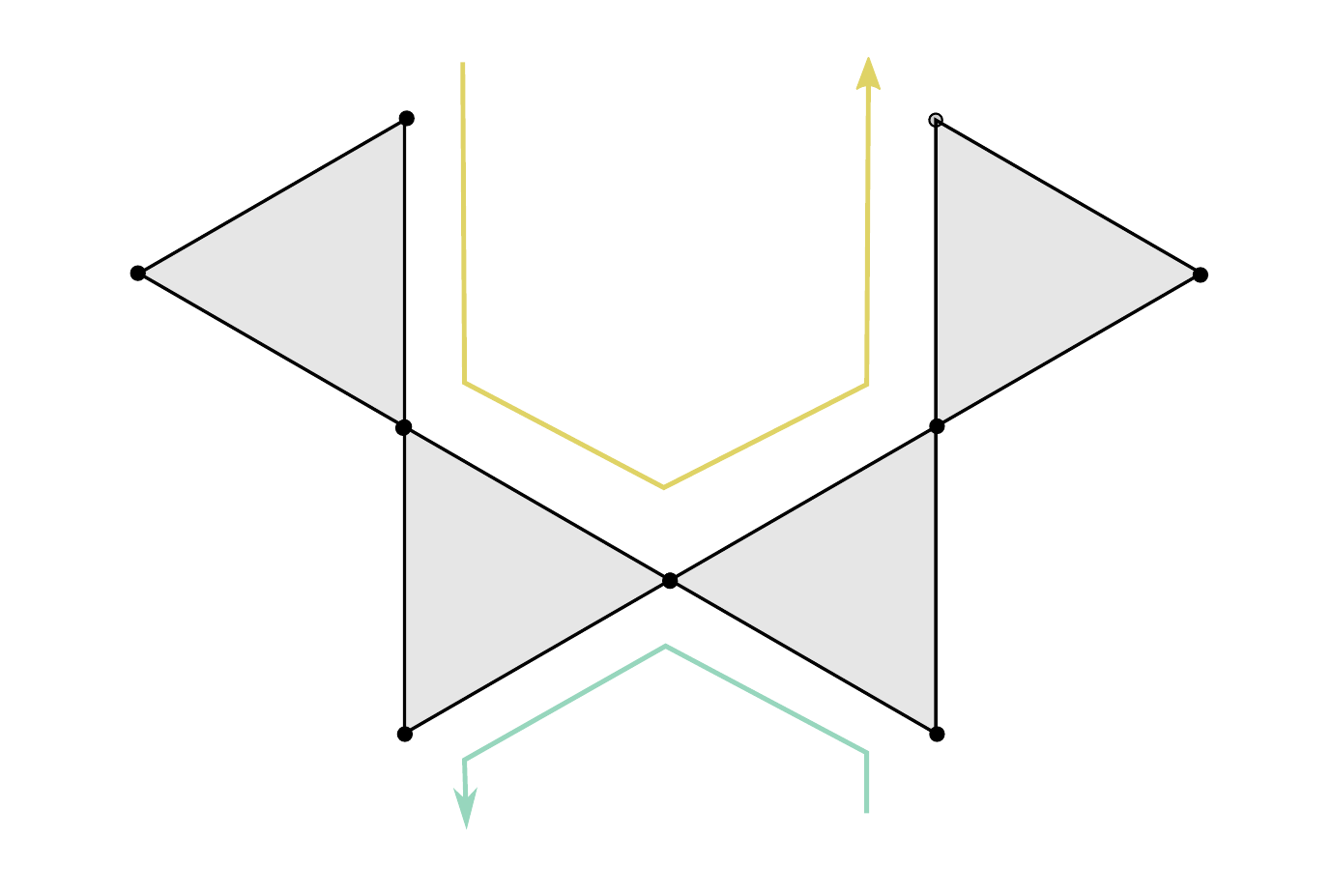}
						\put(54,23){$e$}
						\put(73,33){$s$}
						\put(25,33){$ts$}
						\put(73,10){$st$}
						\put(26,10){$t$}
						\put(92,45){$s\cdot st$}
						\put(1,45){$ts\cdot s$}
						\put(70,61){$s\cdot s$}
						\put(23,61){$ts\cdot st$}
		\end{overpic}
		\caption{A portion of the (flag completion of the) quasi-tree $\cT_{st}$ for $m_{st}=3$, with the combinatorial axes~$\mathrm{Axis}(s)$ in yellow and~$\mathrm{Axis}(t)$ in green.}
	\end{center}
\end{figure}
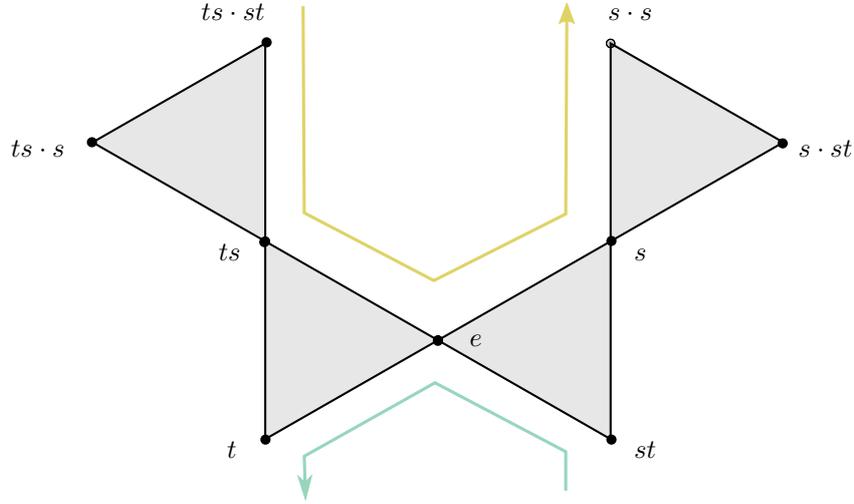

It follows  from the existence and uniqueness of Garside normal forms that elements of    
$A_{st}/\langle \Delta_{st}\rangle $, i.e.  the vertices of $\cT_{st}$, are in $1$-to-$1$ correspondence with left-weighted elements of the free monoid $M^\bullet$ on $M$. The graph $\cT_{st}$ was  introduced (albeit via a different but equivalent definition) and studied in \cite{BestvinaArtin}.  The following fact, already observed in \cite{BestvinaArtin}, is the geometric counterpart of the uniqueness of Garside normal forms in $A_{st}$: 

\begin{lem}\label{lem:garside_quasitree}
	The graph $\cT_{st}$ is a quasi-tree. More precisely, the flag completion of $\cT_{st}$ is a tree of simplices of dimension $m_{st}-1$, whose maximal simplices are either in the $A_{st}$-orbit of the simplex 
	spanned  by $$e, s, st, \ldots, \underbrace{sts\cdots}_{m_{st}-1},$$ or in the $A_{st}$-orbit of the simplex 
	spanned  by $$e, 
	t, ts, \ldots, \underbrace{tst\cdots}_{m_{st}-1}.$$
\end{lem} 

We refer to Figure~2 for an illustration of this result. 

\bigskip

\paragraph{Links of type $2$ vertices.}

Links of type $2$ vertices will play an important role when proving Proposition~C. The following description of these links follow directly from the definition of the Deligne complex.

\begin{lem}\label{lem:link_dihedral_cocompact}
	Let $A_{st}$ be a standard dihedral parabolic corresponding to the vertex~$v_{st}$ of $D_\Gamma$. The link  $\link (v_{st})$ admits the following description: 
	Vertices of $\link (v_{st})$ correspond to cosets of the form $g\langle s\rangle, g \langle t \rangle, $ or $g\{1\}$, 
	and for every $g \in A_{st}$, we add an edge between the vertices $g\{1\}$ and $g\langle s\rangle$, as well as an edge between the vertices $g\{1\}$ 
	and $g \langle t \rangle$. 
	
	In other words, $\link (v_{st})$ is the first barycentric subdivision of the graph whose vertex set is the set of left cosets of $\langle s \rangle$ and $\langle t \rangle$ in $A_{st}$, and such that two cosets are connected by an edge when they have a non-empty intersection.
\end{lem}

\subsection{Standard trees}

Standard trees were introduced in \cite{MP2} and played a key role in the proof of the Tits alternative for two-dimensional Artin groups of hyperbolic type. They will also be used in this article.

\begin{defin}[{\cite[Definition~4.1]{MP2}}]
	Let $A_\Gamma$ be a two-dimensional Artin group, let $a\in A_\Gamma$ be a standard generator, and let $g \in A_\Gamma$. Then the fixed-point set $\Fix(gag^{-1}) = g\Fix(a)$ is a convex subtree of $D_\Gamma$ called a \textbf{standard tree}.
\end{defin}

We recall here a few elementary facts about fixed-point sets and standard trees. For some concrete examples and  illustrations of standard trees, both bounded and unbounded, we refer the reader to \cite[Example~4.7]{MP2}.

\begin{lem}[{\cite[Corollary~2.16]{HMSArtin}}] \label{lem:standard_tree_power}
	Let $a$ be a standard generator of $A_\Gamma$ and let $k \in \Z$ be a non-zero integer. Then the trees $\Fix(a)$ and $\Fix(a^k)$ coincide.
\end{lem}

\begin{lem}[{\cite[Corollary~2.18]{HMSArtin}}] \label{lem:standard_tree_intersect}
	Two distinct standard trees of $D_\Gamma$ are either disjoint or intersect along a single vertex (necessarily of type $2$).
	
	In particular, an element of $A_\Gamma$ globally stabilises a standard tree $T$ if and only if it sends at least one  edge of $T$ to an edge of $T$.
	\end{lem}

\begin{lem}[{\cite[Lemma~2.27]{HMSArtin}}] \label{lem:standard_tree_stab}
	Let $T$ be the standard tree corresponding to the standard generator $a \in A_\Gamma$. 
		The pointwise stabiliser of $T$ is the cyclic subgroup $\langle a \rangle$, and in particular coincides with the stabiliser of every edge of $T$.
		Moreover, the global stabiliser of $T$ is the centraliser $C(a)$ of $a$, which is isomorphic to a direct product of the form $\Z \times F_n$. 
\end{lem}

Finally, we will need a local understanding of standard trees.  By Lemma~\ref{lem:garside_generators}, conjugation by the Garside element $\Delta_{st}$ either permutes the generators $s, t$ (when $m_{st}$ is odd) or centralises each of them (when 
$m_{st}$ is even). Thus, the action of $\langle \Delta_{st} \rangle$ by \textit{right} multiplication on $A_{st}$ induces an action on the set of
\textit{left cosets} of $\langle s \rangle$ and $\langle t \rangle$: 
$$g\langle s \rangle \Delta_{st} = g \Delta_{st} \langle t \rangle ~~ \mbox{ and }  ~~ g\langle t \rangle \Delta_{st} = 
g \Delta_{st} \langle s \rangle  ~~\mbox{ if $m_{st}$ is odd,}$$
$$g\langle s \rangle \Delta_{st} = g \Delta_{st} \langle s \rangle ~~ \mbox{ and }  ~~ g\langle t \rangle \Delta_{st} = 
g \Delta_{st} \langle t \rangle  ~~\mbox{ if $m_{st}$ is even.}$$

The following lemma provides a description of the neighbourhood of a type $2$ vertex in a given standard tree. It is essentially a reformulation of \cite[Lemma~4.3]{MP2}:

\begin{lem}[{\cite[Lemma~2.35]{HMSArtin}}]\label{lem:garside_orbit}
	Let $v_{st}$ be the type $2$ vertex of $D_\Gamma$ corresponding to the coset $A_{st}$. Two vertices of the link $\link(v_{st})$ correspond to edges of $D_\Gamma$ in the same standard tree if and only if the 
	corresponding cosets are in the same $\langle \Delta_{st} \rangle$-orbit (for the multiplication on the right).
\qed
\end{lem}

\section{Free subgroups generated by elliptic elements}\label{sec:elliptic}

The goal of this section is to prove Proposition~C, on the subgroups generated by large powers of elliptic elements with disjoint fixed-point sets. 
Recall from the introduction that our goal is to construct a tree in $D_\Gamma$ on which $\langle a^n, b^n \rangle$ acts, which we obtain by first considering a goedesic $\gamma$ between the fixed-point sets of $a$ and $b$ respectively, and showing that for some large enough $n$, the geodesic segment $\gamma$ makes an angle at least $\pi$ with any of its $\langle a^n \rangle$-translates or $\langle b^n \rangle$-translates. (The notion of angle used here is the standard notion of Alexandrov angle of a CAT(0) space, see for instance \cite[II.1.12]{BH}.)

We will need the following result:

\begin{lem}\label{lem:existence_min_geodesic}
	Let $X$ be a piecewise-Euclidean (resp. piecewise-hyperbolic) complex with finitely many isometry types of simplices such that the associated path-metric is CAT$(0)$ (resp. CAT$(-1)$), and let $C, C'$ be two convex subcomplexes of $X$. Then there exists a geodesic (resp. a unique geodesic) that realises the distance between $C$ and $C'$.
\end{lem}

\begin{proof}
	This is essentially Lemma~\cite[Lemma~5.1]{MP2}. Note that the statement there was only given in the case of piecewise-hyperbolic complexes, but the exact same proof carries over to the piecewise-Euclidean case to prove the existence of a geodesic realising the distance. (The piecewise-hyperbolic structure was only used to prove the uniqueness of such a geodesic.)
\end{proof}

We recalled in Section~\ref{sec:Artin} that the action on the Deligne complex is without inversion, i.e. the global stabiliser of a simplex of $D_\Gamma$ is equal to the pointwise stabiliser of that simplex. In particular,  fixed-point sets of elliptic elements  are convex sub-complexes of the CAT$(0)$ space $D_\Gamma$, and $D_\Gamma$ has finitely many isometry types of simplices by construction. Thus, the above lemma can be applied to to subcomplexes $C, C'$ that are fixed-point sets of elements of $A_\Gamma$ acting elliptically on~$D_\Gamma$. 

\medskip

The key lemma in proving Proposition~C is the following:

\begin{lem}\label{lem:big_angle}
	Let $a, b$ be two elliptic elements with disjoint fixed-point sets. Let $x \in \mathrm{Fix}(a)$ and $y \in \mathrm{Fix}(b)$ be points realising the distance between $\mathrm{Fix}(a)$ and $\mathrm{Fix}(b)$, and let $\gamma$ be the geodesic segment between $x$ and $y$.  Then there exists an integer $n\geq 1$ such that the $\langle a^n \rangle$-translates of $\gamma$ all make an angle of at least $\pi$ with each other at $x$ and $\langle b^n \rangle$-translates of $\gamma$ all make an angle of at least $\pi$ with each other at $y$.
\end{lem}

Let us explain how this lemma implies Proposition~C.

\begin{proof}[Proof of Proposition~C]
	Let $a, b$ be two elliptic elements with disjoint fixed-point sets. Let $x \in \Fix(a)$ and $y \in \Fix(b)$  be points that realise the distance between $\Fix(a)$ and $\Fix(b)$, and let $\gamma$ be the geodesic segment between $x$ and $y$. By Lemma~\ref{lem:big_angle},
	 we can choose an integer $n\geq 1$ such that the $\langle a^n \rangle$-translates of $\gamma$ all make an angle of at least $\pi$ with each other at $x$ and $\langle b^n \rangle$-translates of $\gamma$ all make an angle of at least $\pi$ with each other at $y$.
	
	Let $g \in F(a^n,b^n)$ be a non-trivial reduced word on $a^n$ and $b^n$, which we write in the form $g = g_0\cdots g_k$, with $k\geq 1$, where 
	\begin{itemize}
		\item $g_0 =1$, 
		\item each subword $g_i$ with $i\geq 1$ is a non-trivial power of either $a^n$ or $b^{n}$,
		\item no two consecutives subwords $g_i$ and $g_{i+1}$ are powers of the same  letter $a$ or $b$.
	\end{itemize} 
	To this decomposition $g = g_0 \cdot g_1\cdots g_k$, we associate the subset 
	$$\gamma_g \coloneqq \bigcup_{0\leq i \leq k} (g_0\cdots g_i)\gamma$$
	It follows from the previous discussion that this concatenation of geodesic segments makes an angle at least $\pi$ at each branching point. Since the Deligne complex is CAT$(0)$ by Theorem~\ref{thm:Charney_Davis}, it follows from \cite[II.1.7(a) and II.4.14(2)]{BH} that $\gamma_g$ is a geodesic segment obtained by concatenating $k+1$ geodesic segments, starting with $\gamma$ and ending with $g\cdot \gamma$. In particular $g\cdot \gamma \neq \gamma$.  Thus, every non-trivial reduced word on $a^n$ and $b^n$ defines a  non-trivial element of $A_\Gamma$, and it follows that   $a^n$ and $b^n$ generate a free subgroup of $A_\Gamma$.
\end{proof}

\begin{rem}\label{rem:existence_loxodromic}
	It should be noted that the above proof can be used to further show that the element $g\coloneqq a^nb^n$ acts loxodromically on $D_\Gamma$, by verifying locally that the following subset $\Lambda_g$ is a geodesic line in $D_\Gamma$ that is an axis of $g$: 
	$$\Lambda_g \coloneqq \bigcup_{i\in \Z} g^i (\gamma \cup b^n\cdot \gamma)$$
	More generally, we leave it to the reader to show, using the same approach, that every element of $\langle a^n, b^n \rangle$ that is not conjugated to a power of $a^n$ or $b^n$ acts loxodromically on $D_\Gamma$ (by performing a similar construction to a cyclically reduced element in the conjugacy class of that element).
\end{rem}

The rest of this section is thus dedicated to the proof of Lemma~\ref{lem:big_angle}. In order to measure angles that two geodesics of $D_\Gamma$ make at at a vertex $v$, we will work with the associated angular metric on the link $\mathrm{Lk}(v)$, see for instance \cite[II.7.15]{BH}. Recall that vertices of $\mathrm{Lk}(v)$ are in bijection with edges of $D_\Gamma$ containing $v$, and edges of $\mathrm{Lk}(v)$ are in bijection with triangles of $D_\Gamma$ containing $v$. Moreover, $\mathrm{Lk}(v)$ is made into a metric graph by giving to an edge of $\mathrm{Lk}(v)$ corresponding to the triangle $\Delta$ of $D_\Gamma$ an angular length equal to the angle of $\Delta$ at $v$ (for the piecewise-Euclidean structure defined in~Theorem~\ref{thm:Charney_Davis}). We refer the reader to \cite[Section~4]{CD} for more details on the angular metrics on the links of Deligne complexes.

Note that since $D_\Gamma$ is a piecewise-Euclidean metric with finitely many isometry types of simplices, it follows from \cite[Theorem~A]{BridsonSemisimple} that an element of $A_\Gamma$ acts either elliptically or loxodromically on $D_\Gamma$. 
The following gives a complete description of those elements that act elliptically:

\begin{lem-defin}\label{lem:dichotomy_elliptic}
	Let $g\in A_\Gamma$ be a non-trivial element that acts elliptically on $D_\Gamma$. Then exactly one of the following holds:
	\begin{itemize}
		\item $\Fix(g)$ is a standard tree of $D_\Gamma$. This happens precisely when $g$ is conjugated to a non-trivial power of a standard generator of $A_\Gamma$.  Such an element is called \textbf{tree-elliptic}.
		\item $\Fix(g)$ is a single vertex of type $2 $ of $D_\Gamma$. Such an element is called \textbf{vertex-elliptic}.
	\end{itemize}
\end{lem-defin}

\begin{proof}
	Since $D_\Gamma$ is CAT$(0)$ and $g$ is elliptic, it follows from that $g$ fixes a point of $D_\Gamma$. Since $A_\Gamma$ acts without inversion, $\Fix(g)$ contains a vertex of $D_\Gamma$. Since $g$ is non-trivial, $g$ can only stabilise a vertex
	of type $1$ or $2$. If $g$ stabilises a vertex of type $1$ corresponding to a coset of the form $h\langle s \rangle$, for $s$ a standard generator, then $g$ is conjugated to a power of $hsh^{-1}$, hence $\Fix(g)$ is the standard tree $h\cdot \Fix(s)$ by Lemma~\ref{lem:standard_tree_power}. Otherwise, $g$ only stabilises vertices of type $2$, and by convexity of fixed-point sets in the CAT$(0)$ space $D_\Gamma$, $\Fix(g)$ is a single vertex of type $2$.
\end{proof}

We first case consider the case of a vertex-elliptic element:

\begin{lem}\label{lem:big_angle_1}
	Let $a\in A_\Gamma$ be a vertex-elliptic element, and let $v\in D_\Gamma$ be the corresponding fixed vertex of type $2$. Let $\gamma$ a geodesic segment of $D_\Gamma$ starting at $v$. Then there exists an integer $n\geq 1$ such that the $\langle a^n \rangle$-translates of $\gamma$ all make an angle of at least $\pi$ with each other at $v$. 
\end{lem}

\begin{proof}
	Up to translation by an element of $A_\Gamma$, we can assume that $v$ is the vertex corresponding to the coset $\langle s, t \rangle$. There are two cases to consider depending on whether $m_{s,t}=2$ or $m_{s,t}\geq 3$. 
	
	\medskip
	
\textbf{Case 1:} Let us first consider the case where  $m_{s, t}=2$. In that case, $\langle s, t \rangle \cong \Z^2$.  Vertices of $\mbox{Lk}(v)$ either correspond to type $1$ vertices of the form $s^i \langle t \rangle $ or $t^i \langle s \rangle $ for $i \in \Z$, or to type $0$ vertices of the form $g\langle 1 \rangle$ for $g \in \langle s, t \rangle$. Note that a type $0$ of the form $g \langle 1 \rangle$ has valence $2$ in $\mbox{Lk}(v)$ and is adjacent to $g\langle s \rangle$ and $g \langle t \rangle$. In particular, $\mbox{Lk}(v)$ is isomorphic to the barycentric subdivision of a  graph whose vertices correspond to the type $1$ vertices of $\mathrm{Lk}(v)$. For the remainder of this case, we will for simplicity `forget' about type 0 vertices and will instead identify $\mathrm{Lk}(v)$  with the aforementioned graph. Note that by construction of the Deligne complex, triangles of $D_\Gamma$ (and hence edges of $\mathrm{Lk}(v)$) do not contain two vertices that correspond to left cosets of $\langle s\rangle $ or  two vertices that correspond to left cosets of $\langle t\rangle $, so this graph is  bipartite for the  colouring given by the type of coset (i.e. coset of $\langle s \rangle$ or coset of $\langle t \rangle$ respectively). Moreover, since $\langle s, t \rangle \cong \Z^2$ and $\mathrm{Lk}(v)$ contains the edge between $\langle s \rangle$ and $\langle t \rangle$, it follows that this graph is a complete bipartite graph. (See Figure~\ref{fig:link1} for an illustration.)

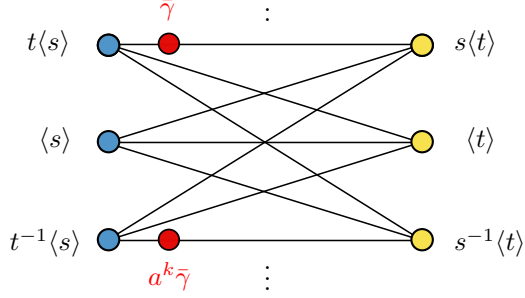
\begin{figure}[H]
	\begin{center}
		\begin{overpic}[width=0.6\textwidth]{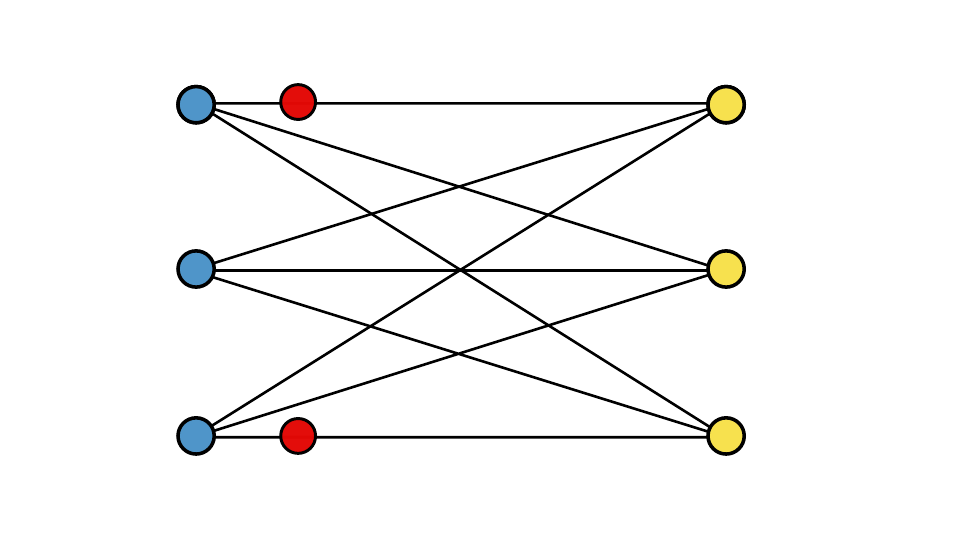}
			\put(8,28){$\langle s \rangle$}
				\put(6,45){$t\langle s \rangle$}
				\put(3,10){$t^{-1}\langle s \rangle$}
				\put(82,28){$\langle t \rangle$}
				\put(80,45){$s\langle t \rangle$}
			\put(80,10){$s^{-1}\langle t \rangle$}
				\put(47,3){$\vdots$}
				\put(47,50){$\vdots$}
			\put(29,51){\color{red} $\bar\gamma$}
			\put(27,4){\color{red} $a^{k}\bar\gamma$}
%			\put(27,31){$a^{2n}b^n\gamma$}
%			\put(1,22){$a^{2n}b^{2n}\gamma$}
		\end{overpic}
	\caption{The link $\mbox{Lk}(v)$, as a complete bipartite graph, with the type 0 vertices ``removed''. Each edge has length $\pi/2$ for the angular metric on $\mathrm{Lk}(v)$.}\label{fig:link1}
	\end{center}
\end{figure}

By construction of the piecewise-Euclidean metric on $D_\Gamma$ and the angular metric on $\mathrm{Lk}(v)$, all egdes of $\mathrm{Lk}(v)$ have length $\pi/2$. In particular, $\mathrm{Lk}(v)$ has diameter exactly~$\pi$.   
Note that $\mathrm{Stab}(v)$ acts on $\mathrm{Lk}(v)$ by preserving the colouring. Moreover, since $a$ is vertex-elliptic, it is not conjugated to a power of the standard generators $s, t \in A_\Gamma$. It follows that $a$ sends a vertex of $\mathrm{Lk}(v)$ corresponding to a coset of $\langle s \rangle $ to a different coset of $\langle s \rangle$, which is thus at distance exactly $\pi$ (and similarly for vertices of $\mathrm{Lk}(v)$ corresponding to a coset of $\langle t \rangle $).  Let $\bar \gamma \in \mathrm{Lk}(v)$ be the point of $\mathrm{Lk}(v)$ corresponding to $\gamma$, and let $e$ be an edge of $\mathrm{Lk}(v)$ containing $\bar \gamma$. Let $k \in \Z$ be a non-zero integer. Then $a^k \cdot e$ is disjoint from $e$ by the above, and since $a^k$ preserves the colouring, it is straightforward to check that $\bar \gamma$ and $a^k \cdot \bar \gamma$ are at distance exactly $\pi$ (see Figure~3). Thus, any two distinct segments in the $\langle a \rangle$-orbit of $\gamma$ make an angle $\pi$ at $v$.

\medskip

\textbf{Case 2:} Let us now consider the case where $m_{st} \geq 3$.  Let $\bar \gamma \in \mathrm{Lk}(v)$ be the point of $\mathrm{Lk}(v)$ corresponding to $\gamma$. Since $a$ is vertex-elliptic, it is not conjugated to a power of the standard generators $s, t \in A_\Gamma$. By \cite[Proposition~E]{VaskouAH}, this implies that the $\langle a \rangle$-orbit of $\bar \gamma$  is quasi-isometrically embedded in $\mathrm{Lk}(v)$. In particular, we can choose an integer $n \geq 1$ such that any two-points of the $\langle a^n \rangle$-orbit of $\bar \gamma$ are at distance at least $\pi$ in $\mathrm{Lk}(v)$. The result follows.
\end{proof}

	We now consider the case of a tree-elliptic element, with associated standard tree $T$. There will be two cases to consider depending on whether $\gamma$ meets $T$ at a vertex of $T$ or at a point contained in the interior of an edge of $T$. 
	
	\begin{lem}\label{lem:big-angle_2}
		Let $a\in A_\Gamma$ be a tree-elliptic element, and let $T\subset D_\Gamma$ be the corresponding fixed tree. Let $Y$ be a convex subcomplex of $D_\Gamma$, and let $\gamma$ be a geodesic segment that realises the distance between $T$ and $Y$. Then there exists an integer $n\geq 1$ such that the $\langle a^n \rangle$-translates of $\gamma$ all make an angle of at least $\pi$ with each other at $v$. 
	\end{lem}

There are several cases to consider, depending on whether $\gamma$ meets $T$ in the interior of an edge, at a type $1$ vertex, or at a type $2$ vertex. We will first need the following result:

\begin{lem}\label{lem:angle_pi_over_2}
	Let $Y, Y'$ be two convex subcomplexes of $D_\Gamma$, and let $\gamma$ be a geodesic segment realising the distance between $Y$ and $Y'$. Then $\gamma$ makes an angle at least $\pi/2$ at $x \coloneqq \gamma \cap Y$ with an any edge of $Y$ containing $x$. 
\end{lem}

\begin{proof}
	This is a direct consequence of \cite[Proposition~II.2.4]{BH}.
\end{proof}

\begin{lem}\label{lem:big_angle_2a}
	Let $a\in A_\Gamma$,  $T\subset D_\Gamma$, $Y \subset D_\Gamma$, and $\gamma$ be as in the statement of Lemma~\ref{lem:big-angle_2}. Suppose further that $\gamma$ meets $T$ at a point $x$ in the interior of an edge $e$ of $T$.  Then any two distinct $\langle a\rangle$-translates of $\gamma$  make an angle of exactly $\pi$ with each other at $x$. 
\end{lem}

\begin{proof}
 By assumption, there exists a triangle $\sigma$ of $D_\Gamma$ containing $e$ such that some non-trivial initial sub-segment of $\gamma$ is contained in $\sigma$.
We know from Lemma~\ref{lem:standard_tree_power} that  $\Fix(a)  = \Fix(a^k)$ for all non-zero integer $k$. Since in addition we have that $\Fix(a)$ is the standard tree $T$, we have $a^k\cdot \sigma \neq \sigma$ for all non-zero integers $k$. 
 For every integer $k$, $a^k\gamma$ realises the distance between $T$ and $a^kY$,
  so $a^k\cdot \gamma$ makes an angle of precisely $\pi/2$ with $e$ by Lemma~\ref{lem:angle_pi_over_2}.  (See Figure~\ref{fig:rightangle} for an illustration.)
  
  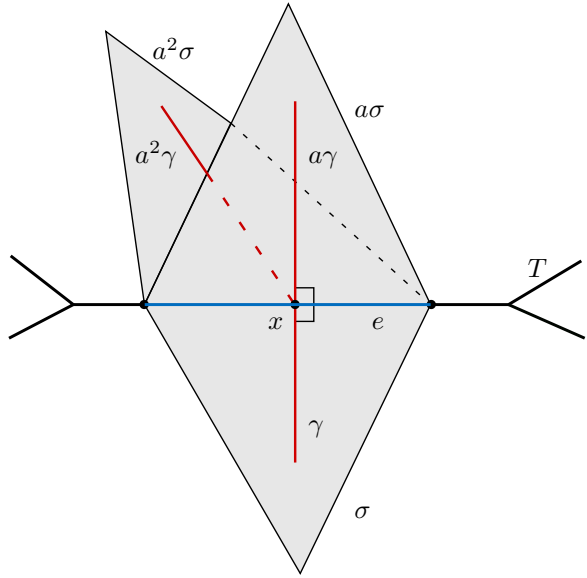
\begin{figure}[H]
  	\begin{center}
  		\begin{overpic}[width=0.6\textwidth]{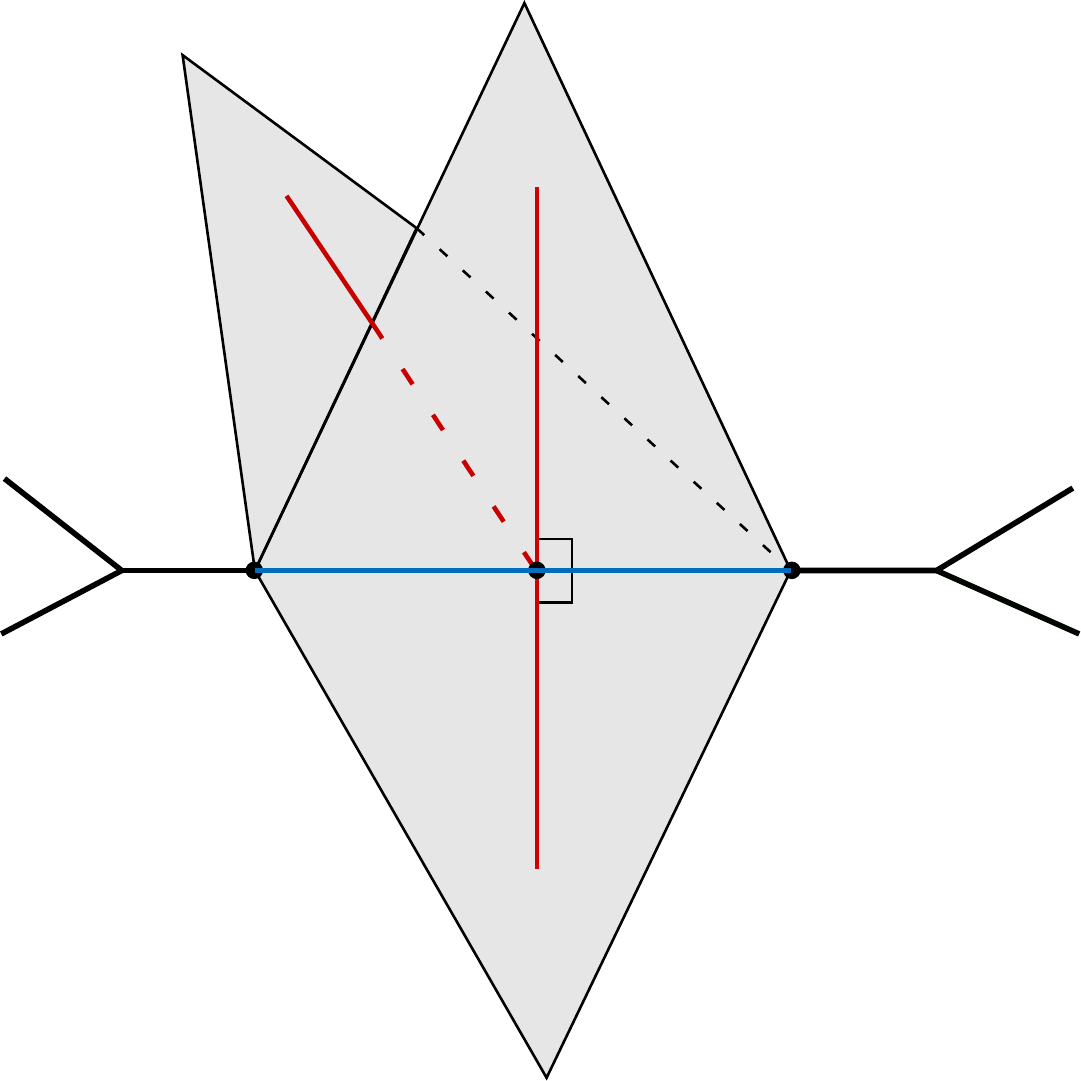}
  			\put(60,10){$\sigma$}
  			\put(60,80){$a\sigma$}
			\put(25,90){$a^2\sigma$}
  			\put(52,25){$\gamma$}
	\put(52,72){$a\gamma$}
	\put(22,72){$a^2\gamma$}
  			\put(45,43){$x$}
  			\put(90,52){$T$}
  			\put(63, 43){$e$}
  		\end{overpic}
  		\caption{The triangles $\sigma, a\sigma, a^2\sigma$ sharing the edge $e$ (in blue) of the standard tree $T$, together with the geodesic segments $\gamma, a\gamma, a^2\gamma$ (in red). Any  translate $a^k\gamma$ of $\gamma$, with $k \neq 0$,  makes an angle exactly $\pi/2$ with $e$ at $x$.}\label{fig:rightangle}
  	\end{center}
  \end{figure}

 Thus, for every non-zero integer $k$, since $\angle_x(e, \gamma) = \pi/2$ and $\angle_x(e, a^k\gamma) = \pi/2$, it follows that  $\angle_x(\gamma, a^k\gamma) \leq \pi$.  Note that the angle between $\gamma$ and $a^k\cdot \gamma$ at $x$ cannot be less than $\pi$, for otherwise the link $\mbox{Lk}(x)$ would contain an embedded loop of length strictly less than $2\pi$ by the above, contradicting the fact that $D_\Gamma$ is CAT$(0)$. We thus have that for every $k \neq 0$, $\gamma$ and $a^k \cdot \gamma$ make an angle exactly $\pi$ at $x$, hence for every distinct integers $p \neq q$,  $a^p \cdot \gamma$ and $a^q \cdot \gamma$ make an angle  $\pi$ at $x$. 
\end{proof}

\begin{lem}\label{lem:big_angle_2b}
	Let $a\in A_\Gamma$,  $T\subset D_\Gamma$, $Y \subset D_\Gamma$, and $\gamma$ be as in the statement of Lemma~\ref{lem:big-angle_2}. Suppose further that $\gamma$ meets $T$ at a type $1$ vertex $v$.  Then there exists an integer $n\geq 1$ such that the $\langle a^n \rangle$-translates of $\gamma$ all make an angle of at least $\pi$ with each other at $v$. 
\end{lem}

\begin{proof}
 Up to translation by an element of $A_\Gamma$, we can assume that $v$ is a vertex of the form $\langle s \rangle$ for some standard generator $s \in V(\Gamma)$, and  $a$ is a power of $s$. By construction of $D_\Gamma$, the vertices of $\mbox{Lk}(v)$ are of two types: Type $0$ vertices of the form $s^i \langle 1 \rangle$ for $i \in \Z$, and type $2$ vertices of the form $\langle s, t \rangle$ for $t\in V(\Gamma)$ adjacent to $s$. Moreover, $\mbox{Lk}(v)$ has the structure of a complete bipartite graph with respect to the colouring by type. By construction of the piecewise-Euclidean structure on $D_\Gamma$ and of the angular metric on $\mbox{Lk}(v)$, all edges of $\mbox{Lk}(v)$ have length $\pi/2$, and in particular, $\mbox{Lk}(v)$ has diameter exactly $\pi$. Note that $T \cap \mbox{Lk}(v)$ consists precisely of the type $2$ vertices of $\mbox{Lk}(v)$. In particular, since $\gamma$ makes an angle at least $\pi/2$ at $v$ with all edges of $T$ by Lemma~\ref{lem:angle_pi_over_2}, the only possibility is for $\gamma$ to correspond to a type $0$ vertex of~$\mbox{Lk}(v)$.

In particular, some initial subsegment of $\gamma$ is contained in a triangle $\sigma$ of $D_\Gamma$ containing some edge $e$ of $T$ and $\gamma$ makes an angle $\pi/2$ with $e$ at $v$. The same argument as for Lemma~\ref{lem:big_angle_2a} now allows us to conclude.  
\end{proof}

\begin{lem}\label{lem:big_angle_2c}
	Let $a\in A_\Gamma$,  $T\subset D_\Gamma$, $Y \subset D_\Gamma$, and $\gamma$ be as in the statement of Lemma~\ref{lem:big-angle_2}. Suppose further that $\gamma$ meets $T$ at a a type $2$ vertex $v$.  Then there exists an integer $n\geq 1$ such that the $\langle a^n \rangle$-translates of $\gamma$ all make an angle of at least $\pi$ with each other at $v$. 
\end{lem}

This final case is the most technical one, and will occupy the remainder of this section. Before proving this lemma, we introduce some new graphs that will be used during the proof. 
We first explain how the link $\mbox{Lk}(v)$ can be reconstructed from a variation of the Cayley graph $  \mathrm{Cay}(A_{st}, S)$, where $S= \{s, t\}$. 

\begin{defin}\label{def:graph of orbits}
	We denote by $\Axis(s)$ the induced subgraph of $ \mathrm{Cay}(A_{st},S)$ spanned by the powers of $s$, and  by $\Axis(t)$ the induced subgraph of $ \mathrm{Cay}(A_{st},S)$ spanned by the powers of $t$.
	We define a new graph $\cG_{s,t}$ (which we think of as an ``augmented'' Cayley graph) as follows:

	We start from the graph that is the disjoint union of all the translates of the axis of $s$ and $t$ in $\mathrm{Cay}(A_{st},S)$:
	$$\bigg(\bigsqcup_{g\in A_{st}/\langle s \rangle} g\Axis(s) \bigg) \sqcup \bigg(\bigsqcup_{g\in A_{st}/\langle t \rangle} g\Axis(t) \bigg)$$
	Edges in such translates are said to be of \textbf{type A} (for ``axis''). 
	
	Whenever two translates $g\Axis(s)$ and $h\Axis(t)$ intersect in $\mathrm{Cay}(A_{st}, S)$, they necessarily intersect along a unique vertex $w \in \mathrm{Cay}(A_{st}, S)$, since $\langle s \rangle \cap \langle t \rangle = \{1\}$. 
	We then add a new edge of $\cG_{s,t}$  between the corresponding vertices $w\in g\Axis(s)$ and $w\in h\Axis(t)$. 
	Such edges are said to be of \textbf{type I} (for ``intersection'').
\end{defin}

By construction, the graph obtained from the augmented graph $\cG_{s,t}$ by collapsing to a point each of the type I edges  is isomorphic to the Cayley graph $\mathrm{Cay}(A_{st}, S)$. We denote by 
$$pr_{I}: \cG_{s,t} \rightarrow \mathrm{Cay}(A_{st}, S)$$ this projection  map.

Moreover, 
the barycentric subdivision of the graph obtained from $\cG_{s,t}$ by collapsing each type A edge is isomorphic to the link $\link(v_{st})$, via the description given in Lemma~\ref{lem:link_dihedral_cocompact}. We denote by $$pr_{A}: \cG_{s,t}\rightarrow \link(v_{st})$$ this projection map. (Note that this projection is not simplicial.)

\medskip

We can summarise this information via the following diagram:

 \[\begin{tikzcd}
{\mathrm{Cay}(A_{st}, S)} &&& {\cG_{s,t}} &&& {\mathrm{Lk}(v)}
\arrow["{pr_{\mathrm{I}}}", from=1-4, to=1-1]
\arrow["{pr_{\mathrm{A}}}"', from=1-4, to=1-7]
\end{tikzcd}\]

\medskip

Recall that conjugation by the Garside element $\Delta_{s,t}$ permutes the generators $s, t$ of $A_{st}$ by Lemma~\ref{lem:garside_generators}. It follows that there is an action of $\langle \Delta_{s,t}\rangle$ by multiplication on the \textit{right} on each of the graphs $\mathrm{Cay}(A_{st}, S)$, $\cG_{s,t}$, and $\link(v)$ (where in the latter case, $\link(v)$ is thought of as a graph of cosets of $A_{s,t}$ as in Lemma~\ref{lem:link_dihedral_cocompact}). Moreover, the projections $pr_A$ and $pr_I$ are $\langle \Delta_{s,t} \rangle$-equivariant. Let us consider the various quotients: 

\medskip

Recall from Lemma~\ref{lem:garside_quasitree} that the quotient $\mbox{Cay}(A_{s,t}, M)/\langle \Delta_{s,t} \rangle$ is precisely the quasi-tree $\cT_{s,t}$. Thus, the quotient  $\mbox{Cay}(A_{s,t}, S)/\langle \Delta_{s,t} \rangle$  is the subgraph of $\cT_{s,t}$ obtained by removing the edges of $\cT_{s,t}$ labelled by an atom on more than one generator. In particular, $\mbox{Cay}(A_{s,t}, S)/\langle \Delta_{s,t} \rangle$ is also a quasi-tree (and is actually a tree, by uniqueness of the Garside normal form).

\medskip

The quotient graph $ \cG_{s,t} /\langle \Delta_{s,t} \rangle$ can be obtained via a similar construction as the graph $\cG_{s,t}$. Indeed, we define a new graph $\cK_{s,t}$ as follows: Denote by $\Axis(s)$ and $\Axis(t)$ the induced subgraphs of $\cT_{s,t}$ spanned by the elements of $\langle s \rangle$ and $\langle t \rangle$ respectively. We start  from the disjoint union of all the subgraphs of the form $g\Axis(s)$ and $g \Axis(t)$, for $g \in A_{s, t}$, and whenever two axes of $\cT_{s,t}$ of the form $g\Axis(s)$ and $h\Axis(t)$ intersect along a (necessarily unique)  element $x$, we add a new edge of $\cK_{s,t}$ between the corresponding vertices $x\in g\Axis(s)$ and $x\in h\Axis(t)$. We define similarly edges of type A and of type I. This graph $\cK_{s,t}$ is isomorphic to the quotient $\cG_{s,t}/\langle \Delta_{s,t} \rangle$.

 As before, there is  a projection $pr_{\mathrm{I}}: \cK_{s,t} \rightarrow \cT_{s,t}$, obtained by collapsing each edge of $\cK_{s,t}$ that is of type I. Similarly, there is a projection $pr_A: \cK_{s,t} \rightarrow \link(v)/\langle \Delta_{s,t} \rangle$ obtained by collapsing each type A edge.

\medskip

  All these graphs fit into the following commutative diagram:

 \[\begin{tikzcd}
 {\mathrm{Cay}(A_{s,t}, S)} &&& {\cG_{s,t}} &&& {\mathrm{Lk}(v)} \\
 \\
 \cT_{s,t} &&& {\cK_{s,t}} &&& {\mathrm{Lk}(v)/\langle \Delta_{s,t} \rangle}
 \arrow["{pr_{\mathrm{I}}}", from=1-4, to=1-1]
 \arrow["{pr_{\mathrm{A}}}"', from=1-4, to=1-7]
 \arrow["{}", from=1-1, to=3-1]
 \arrow["{}", from=1-4, to=3-4]
 \arrow["{}", from=1-7, to=3-7]
 \arrow["{pr_{\mathrm{I}}}", from=3-4, to=3-1]
 \arrow["{pr_{\mathrm{A}}}"', from=3-4, to=3-7]
 \end{tikzcd}\]
 
 where the vertical arrows correspond to quotienting by the action of $\langle \Delta_{s,t} \rangle$ on the right.
 
 \begin{lem}\label{lem:separating_edge}
 	The graph $\cK_{s,t}$ is a quasi-tree. Moreover, for every vertex $v$ of $\cT_{s,t}$, the preimage $W_v\coloneqq pr_I^{-1}(v)$ is a single edge that separates $\cK_{s,t}$ into at least two connected components. 
 \end{lem}

 \begin{proof}
 	Let $v$ be a vertex of $\cT_{s,t}$ corresponding to the class of the element $h \in A_{st}$. Since any two axes of $\cT_{s,t}$ of the form $g\mathrm{Axis}(s)$ are disjoint, and similarly any two axes of the form $g\mathrm{Axis}(s)$ are disjoint, it follows that $pr_I^{-1}(v)$ is a single edge, namely the edge of $\cK_{s,t}$ joining the axes $h\mathrm{Axis}(s)$ and $h\mathrm{Axis}(t)$. Thus, the fibres of the projection $pr_I: \cK_{s,t} \rightarrow \cT_{s,t}$ are non-empty and uniformly bounded, and it follows that $pr_I$ is a quasi-isometry. Since $\cT_{s,t}$ is a quasi-tree by Lemma~\ref{lem:garside_quasitree}, so is~$\cK_{s,t}$.
 	
 	Moreover, since $\cT_{s,t}$ is a tree of simplices by Lemma~\ref{lem:garside_quasitree},  $v$ is cut-vertex of $\cT_{s,t}$, hence its preimage $W_v$ separates $\cK_{s,t}$. 
 \end{proof}

We are now ready to prove Lemma~\ref{lem:big_angle_2c}:

\begin{proof}[Proof of Lemma~\ref{lem:big_angle_2c}] Up to the action of $A_\Gamma$, we can assume that $v$ is the type $2$ vertex $v_{s,t}$ and $a$ is a power of $s$.
	Fix an integer $n \geq 1$, which will be determined later. Let $k >0$ be a  positive integer, and let $p_{kn}$ be a geodesic path in $\mbox{Lk}(v)/\langle \Delta_{s,t}\rangle$ between $\bar \gamma$ and $a^{kn} \bar \gamma$. 	To show that there exists an integer $n \geq 1$ such that any two distinct  $\langle a^n \rangle$-translates of $\gamma$ make an angle of at least $\pi$ at $v$, it is enough to show that in the quotient graph $\mbox{Lk}(v)/\langle \Delta_{s,t}\rangle$, any two distinct  $\langle a^n \rangle$-translates of the point $\bar \gamma$ are at distance at least $\pi$, which we now prove.

	By Lemma~\ref{lem:garside_orbit}, all vertices of $T\cap \link(v)$ are in the same $\langle \Delta_{s,t} \rangle$-orbit (for the multiplication on the right) as the vertex corresponding to the coset $\langle s \rangle$ (since $a$ is a power of $s$ by assumption). In particular, all these edges of $T$ yield a unique vertex in $\mbox{Lk}(v)/\langle \Delta_{s,t}\rangle$, which we will denote as $\bar T$. Note that the condition that $\gamma$ makes an angle at least $\pi/2$ with every edge of $T$ containing $v$ is equivalent to the following:
	$$d(\bar \gamma, \bar T)\geq \frac{\pi}{2}$$
	where we denote by $d$ the metric on $\mbox{Lk}(v)/\langle \Delta_{s,t}\rangle$ induced from the angular metric on $\link(v)$.

	We consider two cases:

\medskip 
	
	\textbf{Case 1:} Assume that $p_{kn}$ goes through the vertex $\bar T$.  Since $p_{kn}$ is a geodesic we have:
	$$d(\bar \gamma, a^{nk}\bar \gamma) = d(\bar \gamma, \bar T) + d(\bar T, a^{nk}\bar \gamma) = d(\bar \gamma, \bar T) + d(a^{nk}\bar T, a^{nk}\bar \gamma) \geq \frac{\pi}{2} + \frac{\pi}{2} = \pi$$

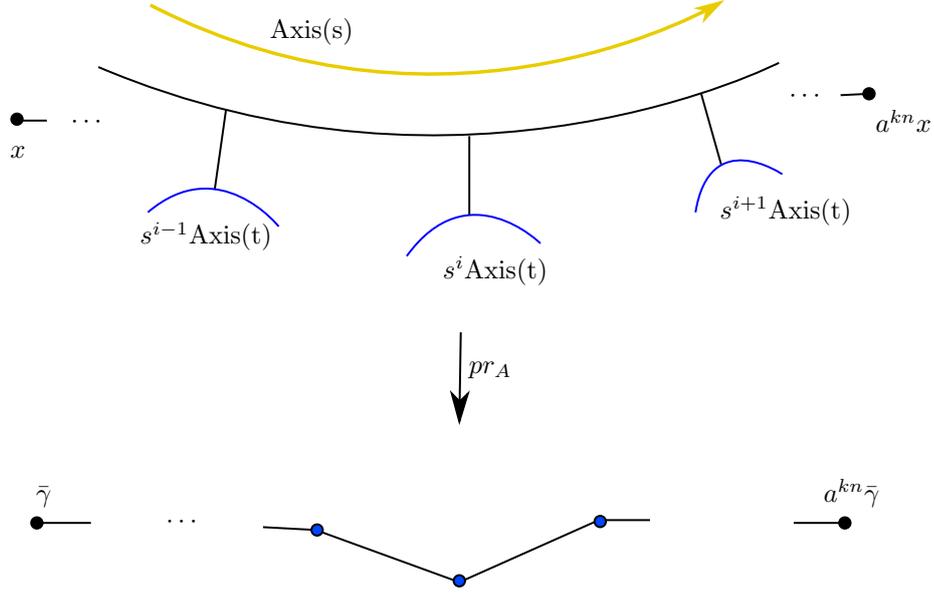
\begin{figure}\label{fig:link2}
	\begin{center}
		\begin{overpic}[width=0.9\textwidth]{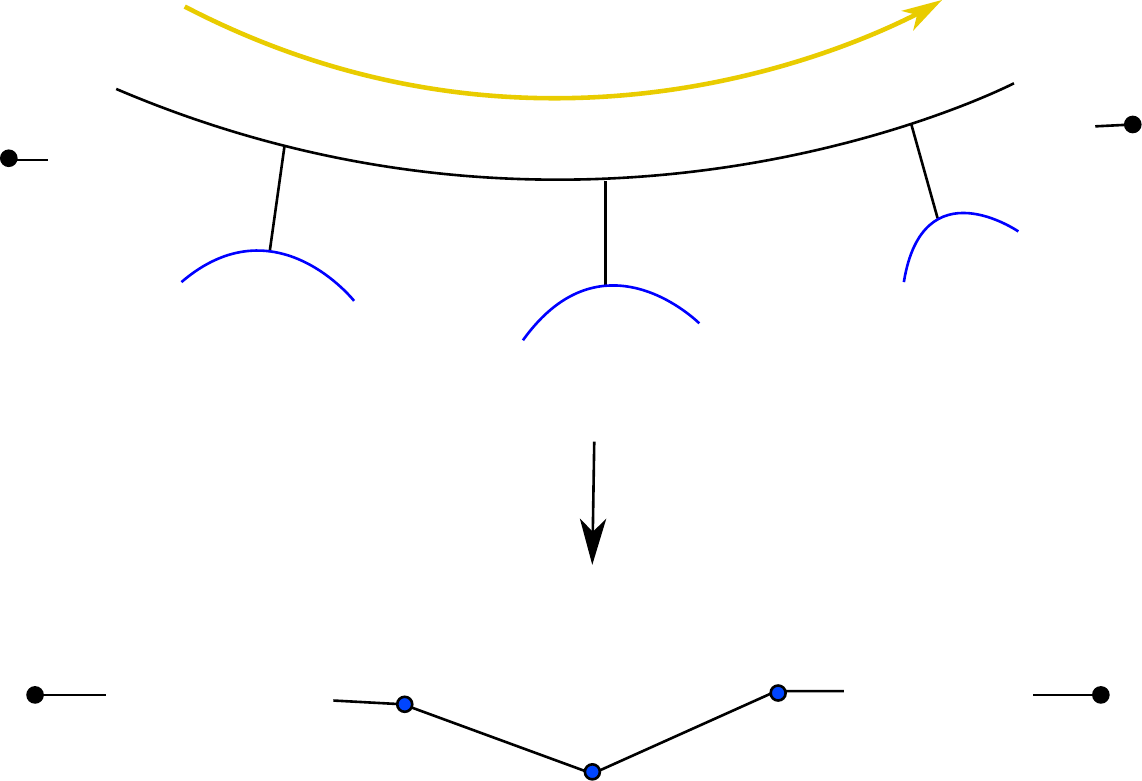}
			\put(0,50){$x$}
			\put(100,53){$a^{kn}x$}
			\put(30,64){$\mathrm{Axis(s)}$}
			\put(3,10){$\bar \gamma$}
			\put(15,40){$s^{i-1}\mathrm{Axis(t)}$}
			\put(50, 36){$s^{i}\mathrm{Axis(t)}$}
			\put(82, 43){$s^{i+1}\mathrm{Axis(t)}$}
			\put(94,10){$a^{kn}\bar \gamma$}
			\put(53,25){$pr_A$}
			\put(18,7){$\cdots $}
			\put(7,53.5){$\cdots $}
			\put(90,56.5){$\cdots $}
		\end{overpic}
		\caption{A portion of the path $\widetilde{p}_{kn}$ in $\cK_{st}$ between $x$ and $a^{kn}x$. Below, its projection $p_{kn}$ in $\mathrm{Lk}(v)/\langle \Delta_{st} \rangle$ going through the vertices corresponding to the axes $s^{i-1}\mathrm{Axis}(t)$, $s^{i}\mathrm{Axis}(t)$, and~$s^{i+1}\mathrm{Axis}(t)$.}
	\end{center}
\end{figure}

	\textbf{Case 2:} Let us now assume that we are given two integers $n$ and $k \neq 0$ such that $p_{kn}$ does not go through the vertex $\bar T$. Since the fibres of the projection $pr_A: \cK_{s,t} \rightarrow \link(v)/\langle \Delta_{s,t} \rangle$ are non-empty and connected, we can  lift the path $p_{kn}$ to a path $\widetilde{p}_{kn}$ of $\cK_{s,t}$ between a vertex $x\in \cK_{s,t}$ and its translate $a^{kn}\cdot x$. 
	 Let us denote by $W_{s^i}$ the edge $pr_I^{-1}(s^i)$. Since $a$ acts loxodromically on the quasi-tree  $\cK_{s,t}$, and by translation on the axis $\mathrm{Axis}(s)$ (since $a$ is a power of $s$ by assumption), there exists a function $f: \bbN \rightarrow \bbN$ with $\lim_{n \infty} f(n) = \infty$ such that $x$ and $a^{kn}\cdot x$ are separated by at least $f(n)$ edges of the form $W_{s^i}$. Recall that each edge $W_{s^i}$ connects $\Axis(s)$ with $s^{i}\Axis(t)$. Since $p_{kn}$ does not go through the vertex $\bar T$, it follows that $\widetilde{p}_{kn}$ does not meet the axis $\mathrm{Axis}(s)$ of $\cK_{st}$. Thus, since $x$ and $a^{kn}\cdot x$ are separated by at least $f(n)$ edges of the form $W_{s^i}$, it follows that $\widetilde{p}_{kn}$ meets at least $f(n)$ subgraphs of the form  $s^{i}\Axis(t)$. Since any two such subgraphs of $\cK_{s,t}$ are disjoint, it follows that 
	the path $p_{kn}$ contains at least $f(n)-1$ distinct edges of~$\mbox{Lk}(v)/\langle \Delta_{s,t}\rangle$.

	Since edges of $\mbox{Lk}(v)/\langle \Delta_{s,t}\rangle$ all have length $\pi/m_{st}$ by construction of the piecewise-Euclidean structure and the angular metric on $\mbox{Lk}(v)$, it follows that 
	$$ d(\bar \gamma, a^{nk}\bar \gamma)  \geq (f(n)-1)\frac{\pi}{m_{st}} \geq \pi$$
where the last inequality holds by choosing $n$ larger than some constant $n_0$ that depends only on $a$. 

\medskip

 It thus follows that if $n \geq n_0$ and $k \neq 0$, then any two geodesic segments in the $\langle a^n \rangle$-orbit  of $\gamma$ make an angle of at least $\pi$ with each other at $v$, which concludes the proof. 
\end{proof}

\begin{proof}[Proof of Lemma~\ref{lem:big-angle_2}]: 
		This now follows from  Lemmas~\ref{lem:big_angle_2a}, ~\ref{lem:big_angle_2b} and~\ref{lem:big_angle_2c}. 
\end{proof}

\begin{proof}[Proof of Lemma~\ref{lem:big_angle}]
	This now follows from  Lemmas~\ref{lem:big_angle_1}, and~\ref{lem:big-angle_2}. 
\end{proof}

\section{The Tits Alternative} \label{sec:Tits}

The goal of this section is to prove Theorem~\ref{thm:Tits_dim2}. To that end, we will use the following classification of the subgroup of $A_\Gamma$:

\begin{prop}\label{prop:subgroup_dichotomy}
	Let $H$ be a subgroup of $A_\Gamma$. Then one of the following holds: 
	\begin{itemize}
		\item $H$ contains two elliptic elements with disjoint fixed-point sets,
		\item $H$ is purely loxodromic,
		\item $H$ is contained in a parabolic subgroup of $A_\Gamma$, or
		\item $H$ stabilises some standard tree of $D_\Gamma$.
	\end{itemize}
\end{prop}

We first show how Proposition~\ref{prop:subgroup_dichotomy} implies Theorem~\ref{thm:Tits_dim2}:

\begin{proof}[Proof of Theorem~\ref{thm:Tits_dim2}]
	By Proposition~\ref{prop:subgroup_dichotomy}, we have four cases to consider. If $H$ contains two elliptic elements with disjoint fixed-point sets, $H$ contains a non-abelian free subgroup by Proposition~C. If $H$ is purely loxdromic, $H$ satisfies the Tits Alternative by \cite[Corollary~C]{OP}. If $H$ is contained in a parabolic subgroup, then by Lemmas~\ref{lem:type_vertices} and~\ref{lem:virtually splits} it is contained in a group that is either trivial, infinite cyclic, or virtually isomorphic to $\Z \times F_k$. Since such groups satisfy the Tits Alternative, so does $H$. Finally, if $H$ stabilises some standard tree of $D_\Gamma$, then $H$ is contained in a subgroup  isomorphic to $\Z \times F_k$ by Lemma~\ref{lem:standard_tree_stab}. Since such a group satisfy the Tits Alternative, so does $H$. 
\end{proof}

The rest of this section is thus devoted to the proof of Proposition~\ref{prop:subgroup_dichotomy}.

\begin{defin}
	We say that a subgroup $H$ of $A_\Gamma$ is \textbf{purely elliptic} if all elements of $H$ act elliptically on $D_\Gamma$, \textbf{purely loxodromic} if all non-trivial elements of $H$ act loxodromically on $D_\Gamma$, and \textbf{mixed} otherwise.
\end{defin}

We start by focusing on mixed subgroups of $A_\Gamma$.

\begin{lem}\label{lem:mixed_dichotomy}
	A mixed subgroup of $A_\Gamma$ either contains two elliptic elements with disjoint fixed-point sets or stabilises some standard tree of $D_\Gamma$.
\end{lem}

We start with a few intermediate results:

\begin{lem}\label{lem:loxodromic_vertex_elliptic}
	Let $H$ be a subgroup of $A_\Gamma$. If $H$ contains a vertex-ellipitic element $a$  and a loxodromic element $b$, then $H$ contains two elliptic elements with disjoint fixed-point sets.
\end{lem}

\begin{proof}
	Let $a$ be a vertex-elliptic element of $A_\Gamma$, and let $b$ be a loxodromic element of $A_\Gamma$. Then $bab^{-1}$ is vertex-elliptic with fixed-point set the vertex $b\cdot \Fix(a)$. Since $b$ is loxodromic, the vertices  $b\cdot \Fix(a) $ and $\Fix(a)$ are distinct. Thus, the elliptic elements $a$ and $bab^{-1}$ are two elements of $H$ with disjoint fixed-point sets.
\end{proof}

\begin{lem}\label{lem:distinct_tree_elliptic}
	Let $H$ be a subgroup of $A_\Gamma$. If $H$ contains two tree-elliptic elements with distinct corresponding trees $T, T'$ that intersect non-trivially, then $H$ contains a vertex-elliptic element with fixed-set the vertex $T\cap T'$. 
\end{lem}

\begin{proof}
	Let $a$ and $b$ be two tree-elliptic elements such that the standard trees $\Fix(a) $ and $\Fix(b)$ are distinct and intersect non-trivially. By Lemma~\ref{lem:standard_tree_intersect}, the intersection $\Fix(a)\cap \Fix(b)$ is a single vertex $v$ of type $2$.  Up to translation by an element of $A_\Gamma$, we can assume that the stabiliser of $v$ is a standard parabolic subgroup $\langle s, t \rangle$, and $a, b$ are conjugates of powers of $s$ or $t$ of the form 
	$$a = gx^ng^{-1}, ~~~~~ b = hy^mh^{-1} ~~~~~~~~~ \mbox{ for some } ~~ g, h \in \langle s, t \rangle, ~~ m, n \in \Z, ~~ x, y \in \{s, t \}$$ 
	Let us show that the element $a^mb^{-n}\in H$ is vertex-elliptic with fixed-vertex the vertex $v$. We already have that $a^mb^{-n}$ fixes $v$. By contradiction, if $a^mb^{-n}$ was a tree-elliptic element in the stabiliser of $v$, we would have an equation of the form 
	$$ gx^{mn}g^{-1}hy^{-mn}h^{-1} = kz^\ell k^{-1} ~~~~~\mbox{ for some }  k \in \langle s, t \rangle, ~~ \ell \in \Z, ~~ z \in \{s, t \}$$
	Applying to this equation the homomorphism $\varphi: A_{st} \rightarrow \Z$ sending every standard generator to $1$ would yield $\ell =0$, hence $a^m = b^n$. By Lemma~\ref{lem:standard_tree_power}, this would imply that $$\Fix(a) = \Fix(a^m) =  \Fix(b^n) = \Fix(b),$$
	contradicting the fact that the trees $\Fix(a)$ and $\Fix(b)$ are distinct.
\end{proof}

\begin{proof}[Proof of Lemma~\ref{lem:mixed_dichotomy}]
	Let $H$ be a mixed subgroup of $A_\Gamma$. If $H$ contains  a vertex-elliptic element, then $H$ contains two elliptic elements with disjoint fixed-point sets by Lemma~\ref{lem:loxodromic_vertex_elliptic}. 
	
	If $H$ contains   two tree-elliptic elements with distinct corresponding trees, then either these trees are disjoint or they intersect. In the former case, $H$ contains two elliptic elements with disjoint fixed-point sets. In the latter case, $H$ contains a vertex-elliptic element by Lemma~\ref{lem:distinct_tree_elliptic}, hence two elliptic elements with disjoint fixed-point sets by Lemma~\ref{lem:loxodromic_vertex_elliptic}.

	Otherwise, all tree-elliptic elements of $H$ have the same standard tree $T$. Let $h$ be such a tree-elliptic element of $H$. An element $g \in H$ sends the tree $T = \Fix(h)$ to the tree $\Fix(ghg^{-1})$, which is also equal to $T$ since $ghg^{-1}\in H$. In particular,  $H$ stabilises~$T$. 
\end{proof}

We now study purely elliptic subgroups:

\begin{lem}\label{lem:elliptic_dichotomy}
	A purely elliptic subgroup of $A_\Gamma$ is contained in a parabolic subgroup of $A_\Gamma$.
\end{lem}

\begin{proof}
	Given two non-trivial elements $a, b \in H$, their fixed-point sets must intersect, for otherwise Remark~\ref{rem:existence_loxodromic} guarantees the existence of loxodromic elements in $H$. 
	
	Suppose that $H$ contains a vertex-elliptic element, with fixed vertex denoted $v$. Since $H$ does not contain loxodromic elements, it follows from Proposition~C that every element of $H$ fixes $v$, hence $H$ is contained in a parabolic subgroup.
	
	Suppose now that $H$ contains two  tree-elliptic elements with distinct fixed-trees (which intersect by the above), then $H$ contains a vertex-elliptic element by Lemma~\ref{lem:distinct_tree_elliptic}, and from the above paragraph $H$ is contained in a parabolic subgroup.
	
	The last case to consider is that all non-trivial elements of $H$ are tree-elliptic with the same fixed tree. In that case, $H$ is contained in the stabiliser of any edge of  that tree, and hence $H$ is contained in a parabolic subgroup. 
\end{proof}

\begin{proof}[Proof of Proposition~\ref{prop:subgroup_dichotomy}] A subgroup of $A_\Gamma$ is either purely loxodromic, mixed, or purely elliptic. The result is thus a consequence of Lemmas~\ref{lem:mixed_dichotomy} and~\ref{lem:elliptic_dichotomy}.
\end{proof}

\section{Subgroups generated by two powers} \label{sec:Wise}

In this section, we prove Theorem~\ref{thm:power_subgroup}. For the remainder of this section, $A_\Gamma$ will denote a two-dimensional Artin group of hyperbolic type. 

\subsection{The coned-off Deligne complex} 

The following modification of the Deligne complex was introduced in \cite{MP2} in order to make two-dimensional Artin groups of hyperbolic type act acylindrically on some hyperbolic complex.

\begin{defin}[{\cite[Definition~4.8]{MP2}}] \label{def:cone_off}
	The \textbf{coned-off Deligne complex} $\widehat{D}_\Gamma$ is the simplicial complex obtained from $D_\Gamma$ by coning-off every standard tree of $D_\Gamma$. That is, $\widehat{D}_\Gamma$ is the flag simplicial complex whose $1$-skeleton is obtained from the $1$-skeleton of $D_\Gamma$ by adding a new vertex $v_T$ for every standard tree $T$ of $D_\Gamma$, and by adding an edge between $v_T$ and each vertex of $T$. 
\end{defin}

The main features of this complex are summarised below:

\begin{thm}[{\cite[Proposition~4.8]{MP2}}]  \label{thm:coneoff_CAT}
	Let $A_\Gamma$ be a two-dimensional Artin group of hyperbolic type. Then the coned-off Deligne complex $\widehat{D}_\Gamma$ admits an $A_\Gamma$-equivariant piecewise-hyperbolic metric such that $\widehat{D}_\Gamma$ is a CAT$(-1)$ space.
\end{thm}

\begin{thm}[{\cite[Theorem~A]{MP2}}]  \label{thm:coneoff_acylindrical}
	Let $A_\Gamma$ be a two-dimensional Artin group of hyperbolic type. Then the action of $A_\Gamma$ on the coned-off Deligne complex $\widehat{D}_\Gamma$ is acylindrical. 
\end{thm}

\subsection{The proof of Wise's Power Alternative}

As before, we will split the proof into several cases, depending on whether $a$ and $b$ act elliptically or loxodromically on $D_\Gamma$. We start with an elementary observation: 

\begin{lem}\label{lem:Z_times_Fk}
	A subgroup of the form $\Z\times F_k$, where $F_k$ is a finitely generated free group, satisfies Wise's Power Alternative.
\end{lem}

\begin{proof}
	This is a direct consequence of the fact that right-angled Artin group satisfy Wise's Power Alternative  \cite{Baudisch}. 
\end{proof}

\begin{lem}\label{lem:parabolic_dichotomy}
	Parabolic subgroups of $A_\Gamma$ satisfy Wise's Power Alternative. 
\end{lem}

\begin{proof}
	We know that vertex stabilisers are either trivial (for type 0 vertices), infinite cyclic  (for type 1 vertices), or virtually isomorphic to a product of the form $\Z\times F_k$ (for type 2 vertices, by Lemma~\ref{lem:virtually splits}). The result is clear for stabilisers of type 0 or type 1 vertices, so it is enough to consider the stabiliser  of a type 2 vertex . Since such a stabiliser contains a finite index subgroup of the form $\Z \times F_k$ for some $k \geq 1$, we can choose $n \geq 1$ so that $a^n, b^n$ belong to such a product $\Z \times F_k$. The result now follows from Lemma~\ref{lem:Z_times_Fk}.
\end{proof}

\paragraph{Elliptic-elliptic.} We start with the case of two elliptic elements: 

\begin{lem}\label{lem:power_elliptic}
	Let $a, b \in A_\Gamma$ be two elliptic elements. Then there exists an integer $n \geq 1$ such that $a^n$ and $b^n$ either commute or generate a non-abelian free group.
\end{lem}
\begin{proof}
	If the fixed-point sets $\Fix(a)$ and $\Fix(b)$ are disjoint, then the result follows from Proposition~C. Otherwise, $a$ and $b$ fix a point of the Deligne complex $D_\Gamma$. Since the action is without inversion, $a$ and $b$ fix a vertex, hence they are contained in some parabolic subgroup of $A_\Gamma$. The result now follows from Lemma~\ref{lem:parabolic_dichotomy}.
	\end{proof}

\paragraph{Loxodromic-loxdromic.} We now consider the case of two loxodromic elements. 

\begin{lem}\label{lem:power_loxodromic}
	Let $a, b \in A_\Gamma$ be two loxodromic elements. Then there exists an integer $n \geq 1$ such that $a^n$ and $b^n$ either commute or generate a non-abelian free group.
\end{lem}

\begin{lem}\label{lem:loxodromic_conedoff}
An element $g \in A_\Gamma$ acts loxodromically on $\widehat{D}_\Gamma$ if and only if it does not stabilise a vertex of $D_\Gamma$ or a standard tree of $D_\Gamma$.
\end{lem}

\begin{proof}
	Since $D_\Gamma$ has finitely many isometry types of simplices by construction, an element of $A_\Gamma$ acts either elliptically or loxodromically by \cite[Theorem~B]{BridsonSemisimple}. If an element $g \in A_\Gamma$ acts elliptically on $\widehat{D}_\Gamma$, it must fix a point since $\widehat{D}_\Gamma$ is CAT$(-1)$ by Theorem~\ref{thm:coneoff_CAT}. Since the action is without inversion, it fixes a vertex. This vertex is either a vertex of $D_\Gamma$ or is an apex of a cone over a standard tree $T$. In the former case, $g$ fixes a vertex of $D_\Gamma$. In the latter case, $g$ stabilises the standard tree~$T$.
\end{proof} 

We will need the following result from  CAT$(-1)$ geometry, which is a variant of \cite[Lemma~3.10]{KapovichBook}: 

\begin{lem}\label{lem:CAT_fellowtravel}
	For every triple of constants $r, \ell, \varepsilon >0$, there exists a constant $L>0$ such that the following holds: Let $Y$ be a CAT$(-1)$ space, and let  $\gamma, \gamma'$ be two geodesics of $Y$ that stay $r$-close on a distance at least $L$. Then there are two sub-segments $\gamma_1 \subset \gamma$ and $\gamma_1' \subset \gamma'$ of length at least $\ell$ and which are $\varepsilon$-close. 
\end{lem}

Recall that the Deligne complex $D_\Gamma$ is CAT$(-1)$ by Theorem~\ref{thm:Charney_Davis}, so in particular every loxodromic element admits a unique axis. We will need the following result, which generalises a well-known result about groups acting geometrically on a proper hyperbolic space: 

\begin{lem}\label{lem:axis_dichotomy}
	Let $b, b' \in A_\Gamma$ be two elements acting loxodromically on $D_\Gamma$,  with axes $\Lambda, \Lambda' $ respectively. Then either $\Lambda = \Lambda'$, or the limit sets of $\Lambda$ and $\Lambda'$ are disjoint in~$\partial D_\Gamma$.  (In the latter case, $\Lambda \cap B(\Lambda', r)$ is  bounded for every $r \geq 0$.)
\end{lem}

\begin{proof} We consider two cases, depending on whether $b$ stabilises a standard tree of~$D_\Gamma$. 
	
	\medskip

	\textbf{Case 1:} Let us first assume that $b$ stabilises a standard tree $T$, so in particular $\Lambda$ is contained in $T$.  Suppose that $\Lambda$ and $\Lambda'$ share a point at infinity. Since $D_\Gamma$ has finitely many isometry types of cells, there exists a constant $\varepsilon$ such that the only edges of $D_\Gamma$ contained in the $\varepsilon$-neighbourhood $N_\varepsilon(\Lambda)$ of $\Lambda$ are precisely the edges of $\Lambda$. Since $\Lambda'$ and $\Lambda$ share a point at infinity, it follows from Lemma~\ref{lem:CAT_fellowtravel} that $\Lambda' \cap N_\varepsilon(\Lambda)$ is unbounded, hence $\Lambda' \cap \Lambda$ is unbounded by construction of $\varepsilon$.  
	Since $b, b'$ acts by translation on $\Lambda, \Lambda'$ respectively, there exists at least one edge $e$ of $ \Lambda \cap \Lambda' \subset T$ and integers $n, m \in \Z$ such that $b^n\cdot e = (b')^m \cdot e$. By Lemma~\ref{lem:standard_tree_stab}, the stabiliser of an edge of $T$ is the pointwise stabiliser of the whole tree $T$. We can thus write $b^n = (b')^mh$, where $h$ is an element fixing $T$ pointwise. We thus have 
	$$\Lambda = \mbox{Axis}(b^n) = \mbox{Axis}((b')^mh) = \mbox{Axis}((b')^m) = \Lambda'$$
	
	\textbf{Case 2:} Let us now consider the case where $b$ does not stabilise any standard tree of $D_\Gamma$.  Since $b$ does not stabilise a point of $D_\Gamma$ or a standard tree of $D_\Gamma$, it acts loxodromically on $\widehat{D}_\Gamma$ by Lemma~\ref{lem:loxodromic_conedoff}. Since in addition the action of $A_\Gamma$ on $\widehat{D}_\Gamma$ is  acylindrical by Theorem~\ref{thm:coneoff_acylindrical}, it follows in particular that $b$ is a WPD element for this action. Since the inclusion map $D_\Gamma \hookrightarrow \widehat{D}_\Gamma$ is Lipschitz, it is straightforward to check that $b$ is also a WPD element for the action of $A_\Gamma$ on $D_\Gamma$. 
	
	Suppose that $\Lambda$ and $\Lambda'$ share a point at infinity. Then, the axes $b^k\Lambda'$, $k \geq 0$, also share a common point at infinity. Since $b$ acts on $D_\Gamma$ as a WPD element, let  $L, N\geq 0$ be integers  such that for a sub-segment $\gamma$ of $\Lambda$ of length at least $L$, there are at most $N$ elements  $g\in A_\Gamma$ such that $\gamma$ and $g\gamma$ are at Hausdorff distance at most $1+|b|$.  Since $b^0\Lambda' , \ldots, b^N\Lambda'$ share the same unique point at infinity, we can choose a subsegment $\gamma \subset \Lambda$ of length $L$ that is in the $1$-neighbourhood of each $b^0\Lambda' , \ldots, b^N\Lambda'$ by Lemma~\ref{lem:CAT_fellowtravel}. In particular, we can  pick integers $\ell_0, \ldots, \ell_N$ such that each segment $b^0(b')^{\ell_0}\gamma , \ldots, b^N(b')^{\ell_N}\gamma$ is within Hausdorff distance $1+|b|$ of $\gamma$. (see Figure~\ref{fig:wpd})
	
	\begin{figure}
		\begin{center}
			\begin{overpic}[width=0.9\textwidth]{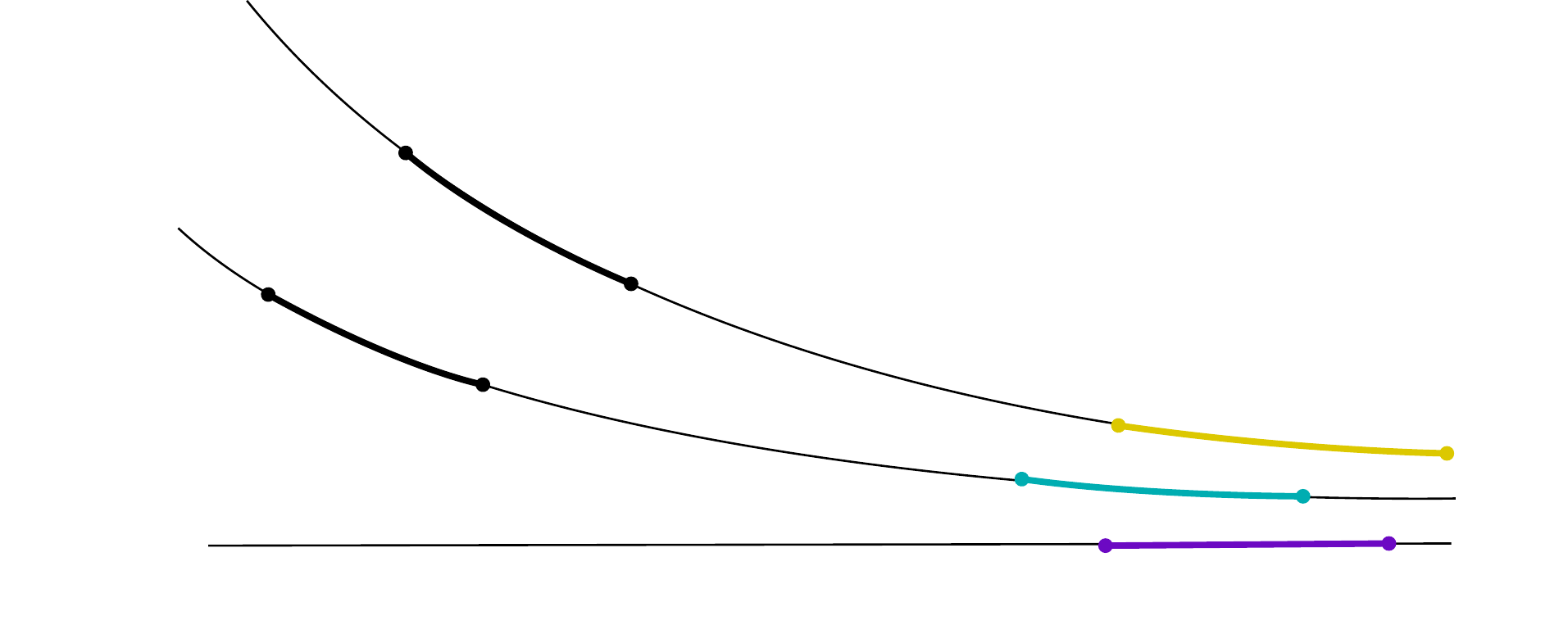}
				\put(95,5){$\gamma$}
				\put(95,8){$b(b')^{\ell_1}\gamma$}
				\put(95,12){$b^2(b')^{\ell_2}\gamma$}
				\put(8,5){$\Lambda'$}
				\put(5,25){$b\Lambda'$}
				\put(9,40){$b^2\Lambda'$}
				\put(30, 30){$b^2\gamma$}
				\put(21,21){$b\gamma$}
			\end{overpic}
			\caption{A portion of the axes $\Lambda'$, $b\Lambda',$ and $ b^2\Lambda'$ together with the geodesic segments $\gamma$ (in purple), $b(b')^{\ell_1}\gamma$  (in green) and~$b^2(b')^{\ell_2}\gamma$ (in yellow).}\label{fig:wpd}
		\end{center}
	\end{figure}
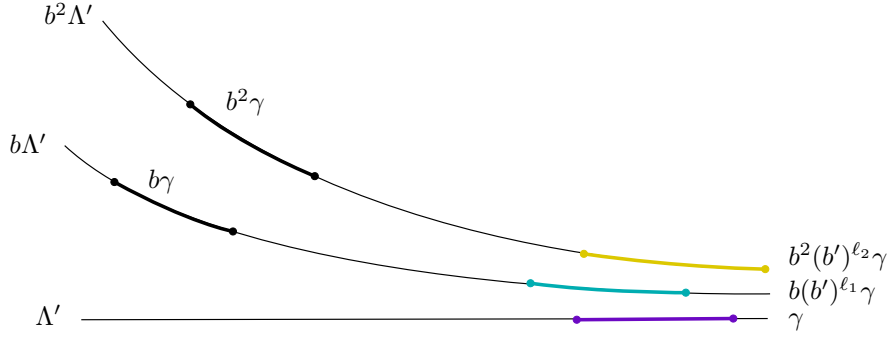

By construction of $L$ and $N$, this means in particular that there exists two integers $j \neq j'$ such that $b^j(b')^{\ell_j} = b^{j'}(b')^{\ell_{j'}}$. Thus, we have an equality of the form $b^p = (b')^q$ for some non-zero integers $p, q$. It now follows that~$\Lambda = \Lambda'$.  
\end{proof}

\begin{proof}[Proof of Lemma~\ref{lem:power_loxodromic}]
	Let $\Lambda, \Lambda'$ denote the axes of $a, b$ respectively. By Lemma~\ref{lem:axis_dichotomy}, we either have $\Lambda = \Lambda'$ or their $\partial \Lambda \cap \partial \Lambda ' = \varnothing$. In the latter case, it is a standard consequence of the North-South dynamics of hyperbolic geometry that sufficiently large powers of $a$ and $b$ generate a free subgroup. 
	
	Let us now assume $\Lambda = \Lambda'$. First consider the case where these axes are contained in a standard tree of the form $\Fix(c)$ for $c$ some conjugate of a standard generator, then both $a$ and $b$ send edges of $\Lambda \subset \Fix(c)$ to edges of $\Lambda \subset \Fix(c)$. It thus follows from Lemma~\ref{lem:standard_tree_intersect} that $a$ and $b$ stabilise $\mathrm{Fix}(c)$, hence belong to $C(c)$ by Lemma~\ref{lem:standard_tree_stab}. Since this subgroup is isomorphic to a direct product of the form $\Z \times F_k$ by Lemma~\ref{lem:standard_tree_stab}, the result now follows from Lemma~\ref{lem:Z_times_Fk}.
	
	Let us now consider the case where the axes $\Lambda = \Lambda'$ are not contained in any standard tree. It follows from Theorem~\ref{thm:coneoff_acylindrical} that $a$ and $b$ act loxodromically on the coned-off space $\widehat{D}_\Gamma$. Moreover, the image of $\Lambda = \Lambda'$ in $\widehat{D}_\Gamma$ defines a quasi-line stable under $\langle b \rangle$ and $\langle b' \rangle$. In particular, the axes of $b$ and $b'$ have the same limit set in $\partial \widehat{D}_\Gamma$, which we still denote $\partial \Lambda = \partial \Lambda'$ for simplicity.  Since $A_\Gamma$ acts acylindrically on the hyperbolic space $\widehat{D}_\Gamma$ by Theorem~\ref{thm:coneoff_acylindrical}, it follows that the stabiliser of $\partial \Lambda$ is virtually cyclic. Thus, there exists an integer $n\geq 1$ such that $a^n$ and $b^n$ commute.
\end{proof}

\paragraph{Loxdromic-elliptic.} We finally consider the mixed case of a loxdromic element and an elliptic element. 

\begin{lem}\label{lem:elliptic_loxodromic}
	Let $a \in A_\Gamma$ be an  element acting elliptically on $D_\Gamma$, and let $b \in A_\Gamma$ be an element acting loxodromically on $D_\Gamma$. Then there exists an integer $n \geq 1$ such that $a^n$ and $b^n$ either commute or generate a non-abelian free group.
\end{lem}

This case is more involved and will require some preliminary results. The key result from hyperbolic geometry we will use is the following: 

\begin{lem}\label{lem:free_subgroup_mixed}
	Let $X$ be a $\delta$-hyperbolic space, let $a$ be a non-trivial elliptic isometry of $X$, and let $b$ be a loxodromic isometry of $X$ with axis $\Lambda$. Suppose that for every $r \geq 0$, there exists a constant $C(r)$ such that for every non-zero integer $k$, the diameter of $a^k\Lambda \cap B(\Lambda, r)$ is at most $C(r)$. Then there exists an integer $n\geq 1$ such that $a$ and $b^n$ generate a free group.  
\end{lem}

\begin{proof}
This is proved by playing ping-pong. Since $b$ acts loxodromically with limit set $\partial \Lambda = \{b^+, b^-\} $ consisting of two distinct points in $\partial X$, we can pick two disjoint neighbourhoods of $b^+, b^-$   in the borderification $X \cup \partial X$, which we call  $U^+, U^-$ respectively. Let $V \coloneqq (X\cup \partial X) - (U^+\cup U^-)$, which we can assume non-empty. 

By the North-South dynamics of loxodromic isometries in hyperbolic spaces, we choose an integer $n\geq 1$ such that for every positive integer $k$ $$b^{kn}\cdot (V\cup  U^+) \subset U^+$$ and $$b^{-kn}\cdot  (V\cup  U^-)  \subset U^-$$ Moreover, the hypothesis  on the $\langle a \rangle$-translates of $\Lambda$ implies that we can further assume that for every non-zero integer $k$, we have $$a^k(U^+ \cup U^-) \subset V$$ 

Let $w$ be a non-trivial reduced word in $a$ and $b^n$. To prove the lemma, it is enough to show that $w$ defines a non-trivial element of $A_\Gamma$. Up to conjugating by a power of $b^n$, we can assume that $w$ starts with a positive power of $b$. If $w$ ends with a non-trivial power of $a$, then the above inclusions imply by ping-pong that $w\cdot U^-  \subset U^+$. If $w$ ends with a non-trivial power of $b$, then the above inclusions imply by ping-pong that $w\cdot V  \subset U^+$. In any case, the word $w$ defines a group element that acts non-trivially on $X$, hence is non-trivial.
\end{proof}

In order to prove Lemma~\ref{lem:elliptic_loxodromic}, we thus need to show that certain translates of the axis of a loxodromic element have ``small coarse overlaps''. We will start by proving the following slightly more general statement:

\begin{lem}\label{lem:axis_overlap}
	Let $\Lambda$ be the axis of a loxodromic element $b\in A_\Gamma$. For every $r>0$, there exists a constant $C\coloneqq C(r, b)$ such that for every $g\in A_\Gamma$, we either have $g\Lambda = \Lambda$ or the diameter of $g\Lambda \cap B(\Lambda, r)$ is at most $C$.
\end{lem}

We split the proof of Proposition~\ref{lem:axis_overlap} into two cases depending on whether $b$ stabilises a standard tree of $D_\Gamma$:

\begin{lem}\label{lem:axis_overlap_1}
	Let $\Lambda$ be the axis of a loxodromic element $b\in A_\Gamma$, and assume that $b$ does not stabilise any standard tree. For every $r>0$, there exists a constant $C\coloneqq C(r, b)$ such that for every $g\in A_\Gamma$, we either have $g\Lambda = \Lambda$ or the diameter of $g\Lambda \cap B(\Lambda, r)$ is at most $C$.
\end{lem}

\begin{proof}
	We know from Theorems~\ref{thm:coneoff_CAT} and~\ref{thm:coneoff_acylindrical} that $A_\Gamma$ acts acylindrically (and cocompactly) on the CAT$(-1)$  space $\widehat{D}_\Gamma$. Moreover,  an element that does not stabilise a point of $D_\Gamma$ or a standard tree of $D_\Gamma$ acts loxodromically on $\widehat{D}_\Gamma$ by Lemma~\ref{lem:loxodromic_conedoff}. Thus, $b$ acts loxodromically on $\widehat{D}_\Gamma$. Since the action of $A_\Gamma$ on $\widehat{D}_\Gamma$ is acylindrical, it follows in particular that $b$ is a WPD element for this action. Since the inclusion map $D_\Gamma \hookrightarrow \widehat{D}_\Gamma$ is Lipschitz, it is straightforward to check that $b$ is also a WPD element for the action of $A_\Gamma$ on $D_\Gamma$. Let $x_0 \in \Lambda$ be a point of the axis $\Lambda$. For an integer $n \geq 1$, we set define the geodesic segment $\gamma_n \coloneqq [x_0, b^nx_0] \subset \Lambda$. Since $b$ is a WPD element for the action of $A_\Gamma$ on $D_\Gamma$, for every $r \geq 0$ there exists an integer $N(r)\geq 1$ such that there exist only finitely many elements $g \in A_\Gamma$ such that $\gamma_{N(r)}$ and $g\gamma_{N(r)}$ are at Hausdorff distance at most $r$ from one another. Let $N$ be the constant associated with $1+|b|$,
	and let $\gamma \coloneqq \gamma_{N}$ be the corresponding geodesic segment. Let us denote by  $g_1, \ldots, g_n \in A_\Gamma$ the finitely many elements of $A_\Gamma$ such that $\gamma_{N}$ and $g_i\gamma_{N}$ are at Hausdorff distance at most $1+|b|$
	from one another. 
	
	By Lemma~\ref{lem:CAT_fellowtravel}, we can choose a constant $L$, which we can assume greater than $|b|$, such that if two geodesic segments of $D_\Gamma$ stay $r$-close for a distance at least $L$, then they stay $1$-close for a distance at least $3|\gamma|$.
	
	\medskip 
	
	Let $g \in A_\Gamma$ be an element and assume  that the diameter of $g\Lambda \cap B(\Lambda, r)$  is at least~$L$. By construction of $L$, this means in particular that $\Lambda$ and $g\Lambda$ $1$-fellow-travel for a distance greater than $3|\gamma|$. Let $\gamma_1 \subset \Lambda$ and $\gamma_2\subset g\Lambda$ be two sub-segments of length $3|\gamma|$ at Hausdorff distance at most $1$. We can now choose integers $k, \ell \in \Z$ such that $b^k\gamma \subset \gamma_1$ and $gb^\ell \gamma \subset \gamma_2$, and such that 
	$b^k\gamma $ and $gb^\ell \gamma$ are within Hausdorff distance $1+|b|$  from one another. (see Figure~\ref{fig:wpd2}.)
	By construction of the $g_i$'s, there exists an integer $1 \leq i \leq n$ such that  $b^{-k}gb^\ell = g_i$. 
	
		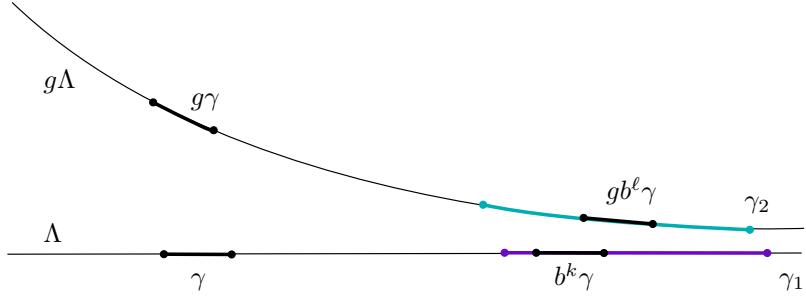
\begin{figure}
		\begin{center}
			\begin{overpic}[width=0.9\textwidth]{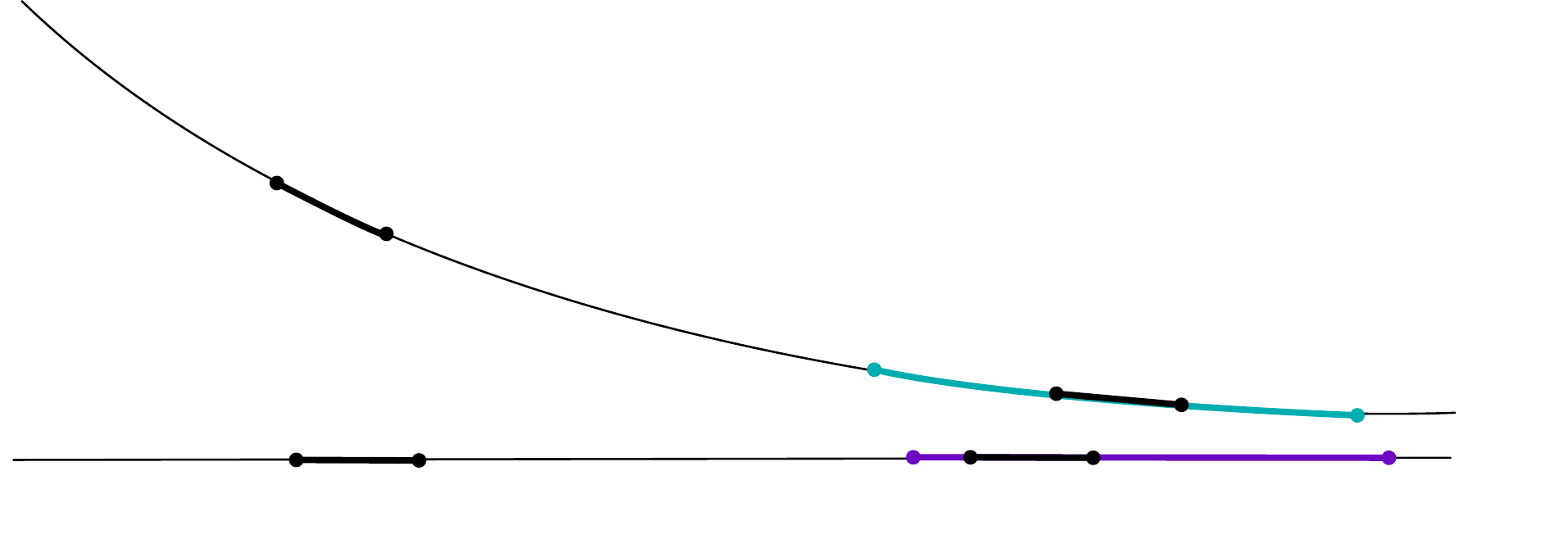}
				\put(5,7){$\Lambda$}
				\put(5,25){$g\Lambda$}
				\put(22,23){$g\gamma$}
				\put(22,2){$\gamma$}
				\put(90,2){$\gamma_1$}
				\put(86,11){$\gamma_2$}
				\put(64,2){$b^k\gamma$}
				\put(70,12){$gb^\ell \gamma$}
			\end{overpic}
			\caption{A portion of the axes $\Lambda$ and  $g\Lambda$. The subsegments $\gamma_1\subset \Lambda$ (in purple) and $ \gamma_2\subset g\Lambda$ (in blue) are at Hausdorff distance at most $1$, and the $\gamma$-translates $b^k\gamma \subset \gamma_1$ and $gb^\ell \gamma\subset \gamma_2$ are at Hausdorff distance at most $1+|b|$.  }\label{fig:wpd2}
		\end{center}
	\end{figure}
	
	For every $1 \leq i \leq n$ such that $g_i\Lambda \neq \Lambda$, we thus  introduce the  constant
	$$C_i\coloneqq \big|g_i\Lambda \cap B(\Lambda, r)\big|,$$
	which is finite by Lemma~\ref{lem:axis_dichotomy}.
	Since $b^{-k}gb^\ell = g_i$ and $b$ acts by translations on $\Lambda$, it follows that $g\Lambda \neq \Lambda$ if and only if $g_i\Lambda \neq \Lambda$. If $g_i\Lambda \neq \Lambda$, we thus get
	$$|g\Lambda \cap B(\Lambda, r)| =  |b^{-k}gb^\ell \Lambda \cap B(\Lambda, r)|  = |g_i\Lambda \cap B(\Lambda, r)|  = C_i$$
	
	It thus follows that $g\Lambda = \Lambda$ or $|g\Lambda \cap B(\Lambda, r)| \leq \mathrm{max}(L, C_1, \ldots, C_n)$.
\end{proof}

\begin{lem}\label{lem:axis_overlap_2}
	Let $\Lambda$ be the axis of a loxodromic element $b\in A_\Gamma$, and assume that $b$ stabilises a standard tree. For every $r>0$, there exists a constant $C\coloneqq C(r, b)$ such that for every $g\in A_\Gamma$, we either have $g\Lambda = \Lambda$ or the diameter of $g\Lambda \cap B(\Lambda, r)$ is at most $C$.
\end{lem}

\begin{proof} We start by proving the following intermediate result:
	
	\bigskip
	
	\textbf{Claim.} For every $r \geq 0$ there exists a constant $C(r)$ such that if $\Lambda$ and $g\Lambda$ stay $r$-close on a distance at least  $C(r)$, then $g$ stabilises the tree $T$. 
	
	\bigskip
	
	Since $D_\Gamma$ has finitely many isometry types of simplices, there exists a constant $\varepsilon$, which depends only on the space $D_\Gamma$, such that the only edges of $D_\Gamma$ contained in the $\varepsilon$-neighbourhood  of $T$ are precisely the edges of $T$. Moreover, let $\alpha$ denote the smallest length of an edge of $D_\Gamma$. 
	
	By Lemma~\ref{lem:CAT_fellowtravel}, we choose a constant $\ell \coloneqq \ell(r, \varepsilon, \alpha)$ such that if two geodesics of $D_\Gamma$ stay $r$-close on a distance at least $\ell$, then they stay $\varepsilon$-close on a distance at least~$3\alpha$.
	
	Suppose that $\Lambda$ and $g\Lambda$ stay $r$-close on a distance at least $\ell$. By construction of $\ell$ and $\varepsilon$ we have that $g\Lambda$ and $\Lambda$ coincide on a sub-path of length at least $3\alpha$, hence $g\Lambda$ and $\Lambda$  share  an edge, by construction of $\alpha$. It follows that $g$ sends some edge of $T$ to some edge of $T$, and  Lemma~\ref{lem:standard_tree_stab} inow implies that $g$ stabilises $T$, proving the Claim.
	
	\bigskip
	
	Let us now assume that $\Lambda$ and $g\Lambda$ stay $r$-close on a distance at least  $C(r)$, as defined in the Claim above. From the above Claim, we have that $g$ acts by isometries on the tree $T$, and it particular it sends the line $\Lambda \subset T$ to the line $g\Lambda = \mathrm{Axis}(gbg^{-1})$.  Since $\Lambda$ and $g\Lambda$ stay $r$-close in $D_\Gamma$ and $T$ is a convex subcomplex of $D_\Gamma$, it follows that $\Lambda$ and $g\Lambda$ stay $r$-close in $T$. 
	Without loss of generality, we can further assume that $C(r)> 2r+2|b|+1$. From the above, it follows that  $\Lambda\cap g\Lambda$ contains a geodesic segment $\gamma$ that is the concatenation of exactly $2|b|+1$ edges. Let us denote by $e$ the central edge of $\gamma$. Then, since $b$ and $gb^{\pm 1}g^{-1}$ are hyperbolic isometries of $T$ with axes $\Lambda$ and $g\Lambda$ respectively and identical translation length $|b|$, we can choose $\alpha \in \{-1, 1\}$ such that $b$ and $gb^{\alpha}g^{-1}$ send $e$ to the same edge of $\gamma$. Thus, the element $h \coloneqq b^{-1}(gb^{\alpha}g^{-1})$ belongs to $\mathrm{Stab}(e)$, hence it acts trivially on $T$ by Lemma~\ref{lem:standard_tree_stab}. In particular, we have that $\mathrm{Axis}(bh) = \Lambda =  \mathrm{Axis}(b)$. Since we have $gb^{\alpha}g^{-1} = bh$,  it follows that 
	$$g\Lambda  = \mathrm{Axis}(gb^{\alpha}g^{-1}) = \mathrm{Axis}(bh) = \Lambda .$$
	Thus, we get $g\Lambda = \Lambda$ or the diameter of $g\Lambda \cap B(\Lambda, r)$ is at most $C(r)$.
\end{proof}

\begin{proof}[Proof of Proposition~\ref{lem:axis_overlap}] This now follows from Lemmas~\ref{lem:axis_overlap_1} and~\ref{lem:axis_overlap_2}.
\end{proof}

We are now ready to prove Lemma~\ref{lem:elliptic_loxodromic}:

\begin{proof}[Proof of Lemma~\ref{lem:elliptic_loxodromic}]
	First assume that the axis $\Lambda$ of $b$ is contained in $\Fix(a)$. In particular, the element $a$ must be tree-elliptic. Since $b$ sends an edge of $\Lambda \subset \Fix(a)$ to an edge of $\Lambda \subset \Fix(a)$, it follows from Lemma~\ref{lem:standard_tree_stab} that $b \in C(a)$. Thus, $a$ and $b$ commute, and we are done.
	
	\bigskip
	
	Let us now consider the case where the axis the axis $\Lambda$ of $b$ is not contained in $\Fix(a)$.  Let us check that   $a^k \Lambda \neq \Lambda$ for all integers $k \neq 0$. Suppose by contradiction that $a^k \Lambda = \Lambda$ for some non-zero integer $k$. Then $a^k$ fixes some point $x \in D_\Gamma$ as $a^k$ acts elliptically on the CAT$(-1)$ space $D_\Gamma$, and it also stabilises the geodesic line $\Lambda$ by assumption. In particular, it follows that $a^k$ fixes pointwise the unique geodesic  from $x$ to $\Lambda$. Thus, $a^k$ stabilises the line $\Lambda$ and fixes a point in it. It now follows that $a^{2k}$ fixes pointwise $\Lambda$, hence $\Lambda \subset \Fix(a^{2k}) = \Fix(a)$, the latter equality following from Lemma~\ref{lem:standard_tree_power}. This was excluded by assumption. 
	
	Thus,  $a^k \Lambda \neq \Lambda$ for all integers $k \neq 0$, and it follows from Lemma~\ref{lem:axis_overlap} that for every $r \geq 0$ there exists a constant $C(r)$ such that for every non-zero integer $k$, the diameter of $a^k\Lambda \cap B(\Lambda, r)$ is at most $C(r)$.  	The result now follows from Lemma~\ref{lem:free_subgroup_mixed}.
\end{proof}

\begin{proof}[Proof of Theorem~\ref{thm:power_subgroup}] This now follows from Lemmas~\ref{lem:power_elliptic}, \ref{lem:power_loxodromic},  and~\ref{lem:elliptic_loxodromic}.
	\end{proof}

\vspace{0.5cm}

\bigskip\noindent
\textbf{Alexandre Martin},

\noindent Address: Department of Mathematics and the Maxwell Institute for the Mathematical Sciences, Heriot-Watt University, Edinburgh EH14 4AS, UK.

\noindent Email: \texttt{alexandre.martin@hw.ac.uk}

\end{document}